\documentclass[18pt,oneside]{elsarticle}
\usepackage{amsfonts,amssymb,latexsym,xspace,epsfig,graphics,color}
\usepackage{amsmath,amsthm,enumerate,stmaryrd,xy}
\usepackage{bbm}
\usepackage{subfigure}
\usepackage{algorithmic}
\usepackage{algorithm}

\bibpunct{(}{)}{,}{a}{,}{;}

\newcommand{\Prob}{\mathbb{P}}

\newcommand{\XU}{[X({\bf U})]}

\newcommand{\xU}{X({\bf U})}

\newcommand{\G}{\ensuremath{\left}}    
\newcommand{\D}{\ensuremath{\right}}    
     
\newcommand{\Z}{\ensuremath{\mathbb{Z}}}    
   
\newtheorem{lemma}{Lemma}[section]

\newtheorem{obs}{Observation}[section]
\newtheorem{defi}{Definition}[section]
\newtheorem{theo}{Theorem}[section]
\newtheorem{prop}{Proposition}[section]
\newtheorem{coro}{Corollary}[section]

\begin{document}
\DeclareGraphicsExtensions{.pdf,.gif,.jpg}
\begin{frontmatter}


\title{Chains with unbounded variable length memory: perfect simulation and visible regeneration scheme}

\author{Sandro Gallo} 
\ead{gsandro@ime.usp.br} 
\address{Universidade de S\~ao Paulo\\ Instituto de Matem\'atica e estat\'istica\\ 1010 Rua do Mat\~ao\\ BP 66281
CEP 
05315-970
  S\~ao Paulo\\ Brasil} 
\tnotetext[t1]{This work is a part of the author's PhD Thesis, written under the supervision of Prof. Antonio Galves and supported by a FAPESP fellowship (grant 2006/57387-0).}

\begin{abstract}
%
%

  We present a new perfect simulation algorithm for stationary chains having unbounded variable length memory. This is the class of infinite memory chains for which the family of transition probabilities is represented by a \emph{probabilistic context tree}. We do not assume any continuity condition: our condition is expressed in terms of the structure of the context tree. More precisely, the length of the contexts  is a deterministic function of the distance to the last occurrence of some determined string of symbols. It turns out that the resulting class of chains can be seen as a natural extension of the class of chains having a renewal string. In particular, our chains exhibit a visible regeneration scheme.
\end{abstract}
\begin{keyword}
Variable length memory chains\sep probabilistic context trees\sep perfect simulation\sep regeneration scheme.
\MSC {60G10 (primary), 60G99 (secondary).}
\end{keyword}


\end{frontmatter}
\section{Introduction}\label{introduction}


We introduce a new class of discrete time stochastic chains  ${\bf X}=(X_{n})_{n\in\mathbb{Z}}$, taking values in a countable alphabet  $A$. These chains have unbounded variable length memory. This means that  the state of the chain at time $0$ depends on an unbounded suffix of the past $\ldots X_{-2}X_{-1}$, whose length  depends on the values assumed by the chain in the past. 
In the present case, the length of this suffix depends on the distance to the last occurrence of a given finite reference string $a_{-i}\ldots a_{-1}$ of symbols of $A$. 
More precisely, there exists a function $f:\mathbb{N}\rightarrow \mathbb{N}$ such that if the last occurrence of $a_{-i}\ldots a_{-1}$ is at distance $k$ in the past, that is if $X_{-k-i}\ldots X_{-k-1}=a_{-i} \ldots a_{-1}$,
and for $j=i,\ldots,k+i-1$ we have $X_{-j}\ldots X_{-j+i-1}\neq a_{-i} \ldots a_{-1}$,
then we need to know $X_{-f(k)-i-k}\ldots X_{-k-i-1}X_{-k-i}\ldots X_{-1}$ in order to decide the state of the chain at time $0$:
\[
\ldots \overbrace{\underbrace{X_{-f(k)-i-k}\ldots X_{-k-i-1}}_{\textrm{length}=f(k)}\underbrace{X_{-k-i}\ldots X_{-k-1}}_{\textrm{last occurrence of}\,\,a_{-i}^{-1}}X_{-k}\ldots X_{-1}}^{\textrm{Suffix of the past we need to know to decide $X_{0}$}}.
\] 
In other words, the family of transition probabilities $P$ for these chains is such that 
\[
P(\cdot|\ldots b_{-2}b_{-1})=P(\cdot|\ldots c_{-2}c_{-1})
\]
whenever the last occurrence of $a_{-i}\ldots a_{-1}$ is at distance $k$ in $\ldots b_{-2}b_{-1}$ and 
\[
b_{-f(k)-1-k}\ldots b_{-2}b_{-1}=c_{-f(k)-1-k}\ldots c_{-2}c_{-1}.
\]
Observe that on $A=\{1,2\}$, if the reference string is the symbol $2$, and the function $f$ is identically $0$, we obtain the renewal chain with symbol $2$ as  renewal symbol. For this reason, we say that this class of stochastic chains generalizes the class of chains having a renewal string. 

We highlight three main parameters for the study of this class of chains: the  size of the reference string, the set of transition probabilities  to the symbols of this reference string, and the deterministic function $f$. 

We ask the following questions: (i) What shall we assume on these parameters in order to guarantee that there exists a stationary chain compatible with such a family of transition probabilities? (ii) Is this stationary chain unique? (iii) What are the statistical properties of this chain? (iv) Does the chain exhibits a regeneration scheme, as in the renewal case?

It is important to observe  that the existing results of the literature on chains of infinite order cannot answer these questions which, at least to our view, are quite natural. The main reason for this is the fact that since the seminal papers of \citet{onicescu/mihoc/1935}, the literature focussed on the so-called continuity assumption, which is not assumed here. In fact, the way we described the family of transition probabilities of our chains fits exactly in the notion of \emph{probabilistic context trees}, introduced by \cite{rissanen/1983}. It follows that the good framework for our study is the one of probabilistic context trees and not that of continuous family. Moreover, there is, so far, no ``well adapted'' (in a sense we will make clear later in this paper) criteria for the existence and uniqueness of the stationary chain compatible with a given probabilistic context tree.

As a consequence of this, the main method we used in order to answer the above questions is the constructive one: we give sufficient conditions on our parameters ensuring that we can perfectly simulate the chain from the stationary distribution. This is our first main result (Algorithms 1 and 2 and Theorem \ref{theo1}). As far as we know, the only perfect simulation algorithm for chains of infinite order, up to date, was the one of \cite{comets/fernandez/ferrari/2002}, which applies in the continuous framework. Our algorithm shares  several features with their algorithm and this is only due to the fact that both algorithms use the \emph{coupling from the past} method (CFTP in the sequel) introduced by  \citet{propp/wilson/1996} to perfectly simulate Markov chains. 

As a byproduct of Theorem \ref{theo1}, we have sufficient conditions for the existence and the uniqueness of the stationary chain (Corollary \ref{coro1}). We also show that this stationary chain has an \emph{hidden} regeneration scheme, and that  the expected size between two consecutive regeneration times is finite (Corollary \ref{coro2}). The denomination \emph{hidden} means that we cannot detect it on the realization of the  chain. This regeneration scheme arises from the perfect simulation algorithm, and is also similar to the one introduced by \cite{comets/fernandez/ferrari/2002}.

The last main result of this paper is the existence, under the same conditions, of  a visible regeneration scheme (Theorem \ref{theo2}). This regeneration scheme can be detected directly on the realization of the chain. But since our chains are not necessarily renewal, detecting the regeneration scheme means, in general, knowing the entire future of the chain.

We would like to emphasize the fact that the continuity assumption has been originally introduced by \cite{doeblin/fortet/1937} as a technical assumption, enabling them to obtain results (see discussion therein). As they say, it is quite natural in this optic to assume that the probability transition from $a_{-\infty}^{-1}$ to $a$ does not depend to much on the remote symbols of $a_{-\infty}^{-1}$. However, continuity is \emph{one} way to mathematically translate this assumption. The present paper gives a different one. 
The  success of the continuity assumption appeared three decades later (in the 60's), in part because  it is related with some well behaved dynamical systems and statistical mechanic models, through the Gibbs formalism (see for example \cite{bowen/2008}). However, from the application point of view, it is not clear that the real phenomena have to be represented through the continuous framework. It seems to us quite natural (and also of  mathematical interest)   to explore the non-continuous world. 

Therefore, the interest of the present work is threefold. First, it extends the class of renewal chains to a class of stochastic chains having a visible regeneration scheme which is not a renewal scheme.  Second, it gives an appropriate condition on the form of the context tree to guarantee the possibility to make a perfect simulation of the unique stationary chain compatible.  Finally, it seems to be the first attempt of the literature of chains of infinite memory considering the non-continuous case.

\vspace{0.1cm}
This paper is organized as follows. Section \ref{basicdef} gives the basic definitions and notation, introducing in particular the context tree framework. In Section \ref{anexample} we give an example, which motivates the above discussion and explains the reason why we have been taken to consider such a class of stochastic chains. Section \ref{ontheform} explains more precisely our assumptions using the context tree framework.  In section \ref{algorithm} we sketch the perfect simulation algorithm and state the results of this paper. Sections \ref{proof of main theorem}, \ref{provatheo1.5} and \ref{proofoftheorem2} are dedicated to the proofs of the results. In Section \ref{realistic}, we present  the complete perfect simulation algorithm, plus some simulations of the example of Section \ref{anexample}.
We terminate the paper with some literature on the areas involved in this paper.

\section{Basic definitions}\label{basicdef}

Let $A$ be a countable alphabet. Given two integers $m\leq n$, we
denote by $a_m^n$ the string $a_m \ldots a_n$ of symbols in
$A$. For any $m\leq n$, the length of the string $a_m^n$ is denoted by $|a_m^n|$ and is
defined by $|a_m^n| = n-m+1$. For any $n\in\mathbb{Z}$, we will use the convention that $a_{n+1}^{n}=\emptyset$, and naturally $|a_{n+1}^{n}|=0$. Given two strings $v$ and $v'$, we 
denote by $vv'$ the string of length $|v| + |v'| $ obtained by
concatenating the two strings. The concatenation of strings is also
extended to the case in which $v$ denotes a semi-infinite sequence,
that is $v=v_{-\infty}^{-1}$. If $n$ is a positive integer and $v$ a finite string of symbols in $A$, we denote by $v^{n}=vv\ldots v$ the concatenation  of $n$ times the string $v$.
We  denote
$$
A^{-\mathbb{N}}=A^{\{\ldots,-2,-1\}}\,\,\,\,\,\,\textrm{ and }\,\,\,\,\,\,\,  A^{\star} \,=\, \bigcup_{j=0}^{+\infty}\,A^{\{-j,\dots, -1\}}\, ,
$$
which are, respectively, the set of all infinite strings of past symbols and the set of all finite strings of past symbols. 
The case $j=0$ corresponds to the empty string $\emptyset$.

\subsection{Probabilistic context tree}

We say that a string $s$ is a \emph{suffix} (resp. \emph{prefix}) of another string
$v$ if $|s|\leq|v|$ and $v_{-|s|}^{-1}=s$ (resp. $v_{-|v|}^{-|v|+|s|-1}=s$). 
\begin{defi} 
  A subset $\tau$ of $A^{\star}\cup A^{-\mathbb{N}}$ is a \emph{tree} if no string
  $s \in \tau$ is a suffix of another string $v \in \tau$. This
  property is called the \emph{suffix property}. 
When we have $\sup\{|v| : v\in\tau\}= +\infty$, we say that the tree $\tau$ is \emph{unbounded}. 
\end{defi}

\begin{defi}
A tree $\tau$ is \emph{complete} if any element $a_{-\infty}^{-1}$ of $A^{-\mathbb{N}}$ has a 
suffix belonging to $\tau$. The suffix property implies that this
suffix is unique. We  call it the \emph{context} of the sequence $a_{-\infty}^{-1}$ and it  is denoted by $c_{\tau}(a_{-\infty}^{-1})$. A complete tree is called a \emph{context tree}.
\end{defi}

We also extend the notion of context for finite strings: for any $a_{m}^{n}\in A^{\star}$, $m\leq n$, we put $c_{\tau}(a_{m}^{n})=v$ if $v$ is a suffix of $a_{m}^{n}$ belonging to $\tau$. If no context of $\tau$ is suffix of $a_{m}^{n}$, we use the convention $c_{\tau}(a_{m}^{n})=\emptyset$. In particular, $c_{\tau}(\emptyset)=\emptyset$.

\begin{defi}\label{def:pct}
A \emph{probabilistic context tree on $A$} is an
  ordered pair $(\tau,p)$ such that
\begin{enumerate}
\item $\tau$ is a context tree;
\item $p = \{p(\cdot|v); v \in \tau\}$ is a family of transition
  probabilities over $A$.
\end{enumerate}
\end{defi}
Examples of probabilistic context trees are shown in Figures \ref{fig:finite-pct} (for the bounded case) and  \ref{fig:infinite-pct} (for the unbounded case). 

We call the attention of the reader on the following notation: if $v=v_{-|v|},\ldots,v_{-1}$ is a context of a context tree $\tau$, then $v_{-i}$, $i=1,\ldots,|v|$ denotes the path from the root to the leaf in the tree representation of $\tau$. In the conditional probability $p(a|v)$, we will swap the order of the symbols of the context $v$ to keep the overall temporal order:
\[
p(a|v)=p(a|v_{-1}\ldots v_{-|v|})
\]
is the probability that at time $n$, say, we put symbol $a$ given that at time $n-1$ we have $v_{-1}$, at time $n-2$ we have $v_{-2}$...

The context tree illustrated in Figure \ref{fig:finite-pct} is defined on $\{1,2,3\}$, and to each context $v$, we assign in the boxes the transition probabilities $p(1|v)$, $p(2|v)$, $p(3|v)$. In this paper we  only consider unbounded context trees, but we put this example to remember that the model has been introduced by \citet{rissanen/1983} in the bounded case. The context tree illustrated in Figure \ref{fig:infinite-pct} is defined on $\{1,2\}$, it corresponds to the discrete renewal chain, with renewal symbol $2$. This means that the successive occurrences of $2$ ``split'' the realization of the chain into i.i.d. block. For any $i\geq0$, $p_{i}$ denotes the transition probability $p(2|1^{i}2)$. The transition probability $p(1|1^{i}2)$ is  $1-p_{i}$.

More general examples of unbounded context trees (without specifying the transition probabilities) are given by Figures \ref{fig:lapartition}, \ref{fig:partition} and \ref{fig:partition3}.


\begin{figure}[h!]
\begin{center}
\subfigure[]{
\includegraphics[scale=0.9]{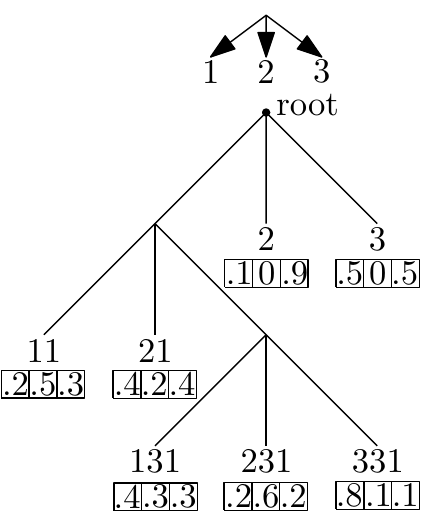}
\label{fig:finite-pct}}
\hspace{1cm}
\subfigure[]{
\includegraphics[scale=1]{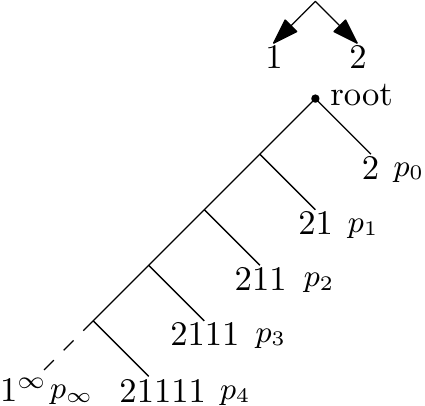}
\label{fig:infinite-pct}}
\caption{Examples of probabilistic context trees.}
\end{center}
\end{figure}

\begin{defi}
We say that a symbol $a$ of $A$ is \emph{$\epsilon$-regular} for a probabilistic context tree $(\tau,p)$ if 
\begin{equation}
\inf_{v\in\tau}p(a|v)\geq\epsilon >0\,.
\end{equation}
A string $w$ of $A^{\star}$ is also said to be $\epsilon$-regular if all symbols of $A$ appearing in $w$ are $\epsilon$-regular. For any  fixed $\epsilon>0$, we denote by $\mathcal{E}$ the set of elements of $A$ which are $\epsilon$-regulars.
\end{defi}

\subsection{Unbounded variable length memory chains}

\begin{defi}
We say that a  stochastic chain ${\bf X}=(X_{n})_{n\in\mathbb{Z}}$ of law $\mathbb{P}$ is compatible with a probabilistic context tree $(\tau,p)$ if for $\mathbb{P}$-a.e. past $a_{-\infty}^{-1}\in A^{-\mathbb{N}}$ and any $a\in A$ we have
\begin{equation}\label{compatible}
\mathbb{P}(X_{0}=a|X_{-\infty}^{-1}=a_{-\infty}^{-1})=p(a|c_{\tau}(a_{-\infty}^{-1})).
\end{equation}
These chains are called \emph{variable length memory chains}. If $\tau$ is unbounded, these chains are called \emph{unbounded} variable length memory chains.
\end{defi}

It remains to translate what we called the ``reference'' string in this framework. This will be done in Section \ref{ontheform}, but first, let us give an example.

\section{Discussion and examples}\label{anexample}

\subsection{Motivating the present work}\label{renewall}

Let us explain why the existing results of the literature cannot help us in our study. First, let us show that our chains need not to be continuous. A  family of transition probabilities $P$ is continuous if
\[
\beta_{k}:=\sup\{|P(a|a_{-\infty}^{-1})-P(a|b_{-\infty}^{-1})|:a\in A,a_{-\infty}^{-1},b_{-\infty}^{-1}\in A^{-\mathbb{N}}, a_{-k}^{-1}=b_{-k}^{-1}\}
\]
converges to zero when $k$ diverges. It is enough to consider the probabilistic context tree on $\{1,2\}$ illustrated in Figure \ref{fig:infinite-pct}, with the following probability transitions:
\[
p_{i}=\epsilon_{1}.{\bf 1}\{i\,\,\textrm{is odd}\}+\epsilon_{2}.{\bf 1}\{i\,\,\textrm{is even}\}\,\,\textrm{and}\,\,p_{\infty}:=p(2|1^{+\infty})=\epsilon_{3}
\]
where $\epsilon_{1}$, $\epsilon_{2}$ and $\epsilon_{3}$ are different real numbers in $(0,1)$. 
Then it is straightforward to check that 
\[
\beta_{k}=\sup\{|\epsilon_{1}-\epsilon_{2}|,|\epsilon_{2}-\epsilon_{3}|,|\epsilon_{1}-\epsilon_{3}|\}
\]
for any $k\geq0$. It follows from this simple observation that none of the chains in the class we consider have to be continuous. 

The second point motivating the present work is that the perfect simulation algorithm given by \cite{comets/fernandez/ferrari/2002} is not well adapted to context trees. This follows from the fact that it does not use the information of the context tree. One more time, in order to illustrate this fact, let us consider the simplest possible case: $(\tau,p)$ is such that
\begin{itemize}
\item $\tau$ is the context tree illustrated in Figure \ref{fig:infinite-pct}
\item $p$ satisfies the continuity condition of \cite{comets/fernandez/ferrari/2002}, and 
\item symbol $2$ is $\epsilon$-regular. 
\end{itemize}
In this case, there exists a very simple procedure to construct a sample $X_{0}^{n}$ of the chain compatible with $(\tau,p)$ form the stationary measure. We use the fact that $2$ is $\epsilon$-regular to couple the constructed chain with an i.i.d. sequence ${\bf U}$ of random variable uniformly distributed in $[0,1[$ in such a way that we put $X_{i}=2$ whenever $U_{i}\leq\epsilon$. We generate backward in time $U_{0}, U_{-1},\ldots$, and stop the procedure at the first time $-k$ such that  $U_{-k}<\epsilon$. We write $-k=\theta[0,n]$, it is a regeneration time for $[0,n]$, and we put $X_{-k}=2$. Then, we construct the sample recursively from time $-k+1$ up to $n$: for any $i\geq k+1$, we sample $X_{i}$ from the distribution $(p(\cdot|X_{i-1}\ldots X_{k})-\epsilon)/(1-\epsilon)$. Observe that $p(\cdot|X_{i-1}\ldots X_{k})$ is always well defined since $X_{k}=2$. The constructed sample is stationary and compatible with $(\tau,p)$. The way we defined $\theta[0,n]$ implies that it has a geometric distribution with parameter $\epsilon$. This procedure is well known in the perfect simulation literature, such a chain is said to be uniformly minorized by $\epsilon$ (see for example \cite{foss/tweedie/1998}). The perfect simulation algorithm we present in this paper works for much more general chains, and is an extension of the procedure we just described. 

Now, suppose we perform the above algorithm and the one of \cite{comets/fernandez/ferrari/2002} at the same time,  using the same sequence ${\bf U}$, to perfectly simulate a window $X_{0}^{n}$. Call $\theta^{CFF}[0,n]$ the regeneration time obtained using their algorithm. The most objective way to compare both algorithms is to check which one is faster. Their random variable $\theta^{CFF}[0,n]$ is defined by
\[
\theta^{CFF}[0,n]=\max\{k\leq0:U_{i}<a_{i-k},\,\,i=k,\ldots,n\},
\]
where $(a_{j})_{j\geq0}$ is a $[0,1]$-valued sequence increasing to $1$. The way it increases depends on the continuity assumption they make. Anyway, to compare here with the above algorithm, we assume that $a_{0}=\epsilon$ if only $2$ is $\epsilon$-regular, and $a_{0}=2\epsilon$ if both symbols are $\epsilon$-regular. If $\theta[0,n]=-k$, it means that $U_{-k}<a_{0}$, and in this case, say, it puts $X_{-k}=2$ if $U_{-k}<\epsilon$ and $X_{-k}=1$ if $\epsilon\leq U_{-k}<2\epsilon$ (if $1$ is also $\epsilon$-regular). The only way their algorithm could be faster than the above algorithm would be that their regeneration time occurs before the first time $U_{-k}\leq\epsilon$:
\[
\theta^{CFF}[0,n]>\max\{k\leq0:U_{j}<\epsilon\}=:\rho.
\]
Therefore, denoting by $\mathbb{P}$ the law of ${\bf U}$
\[
\mathbb{P}(\theta^{G}[0,n]<\theta^{CFF}[0,n])=\mathbb{P}(\rho<\theta^{CFF}[0,n]).
\]
It is difficult to find a good upper bound for this probability in general. We just mention two simple cases. In the case where $\inf_{v\in\tau}p(1|v)=0$, it is clear  that we have $\mathbb{P}(\rho<\theta^{CFF}[0,n])=0$.
The other case we can study easily is when  $p_{i}\searrow p_{\infty}=\epsilon$. A quick look to their algorithm shows that in this case, if $\rho<\theta^{CFF}[0,n]$, then the reconstructed sample is all $1$: $X_{\theta[0,n]}^{n}=1^{|\theta[0,n]|+n+1}$. Let us compute
\[
\mathbb{P}(\rho<\theta^{CFF}[0,n])=\sum_{i\geq0}\mathbb{P}(\rho=-i,\theta^{CFF}[0,n]>-i)
\]
\[
=\sum_{i\geq0}\sum_{j=0}^{i-1}\mathbb{P}(\rho=-i,\theta^{CFF}[0,n]=-j).
\]
Since the event $\{\theta^{CFF}[0,n]=-j\}$ depends only on $U_{-j}^{0}$, it follows that the later term is bounded above by
\[
\sum_{i\geq0}\sum_{j=0}^{i-1}(1-\epsilon)^{i-j+1}\mathbb{P}(X_{-j}^{n}=0^{n+j+1}).
\] 
Using the fact that the symbol $2$ is $\epsilon$-regular, the probability $\mathbb{P}(X_{-j}^{n}=1^{n+j+1})$ is bounded above by $(1-\epsilon)^{n+j+1}$, it follows that for some constant $C>0$
\[
\mathbb{P}(\theta^{G}[0,n]<\theta^{CFF}[0,n])\leq C(1-\epsilon)^{n}.
\]
Which goes very fast to $0$. These facts are a consequence of the following  general remark. Since their algorithm does not use the form of the context tree, it leads to regeneration times that have, \emph{a priori}, nothing to do with the \emph{natural regeneration times} one could expect: the successive occurrences of $2$ along the realization of the chain. This leads to another misleading situation: their regeneration times cannot be seen on the realization of the chain.

\subsection{Example}

As we said, the symbol $2$ is renewal for the chain compatible with the context tree of Figure \ref{fig:infinite-pct}. Therefore,  an example of extension of this model is the following: ``if the last occurrence of $2$ occurred at a distance $i$ in the past (that is, if the last $i$ symbols are all $1$'s), then  look back $i$ sites further this occurrence''.  In other words,
the contexts have the form 
\[
a_{-2i-1}^{-i-2}\,2\,1^{i}\,,\,\,\forall i\geq0\,\,\textrm{and}\,\,a_{-2i-1}^{-i-2}\in A^{i},
\]
the context tree is
\begin{equation}\label{lexamplecontexttree}
\tau=\cup_{i\geq0}\cup_{c\in A^{i}}c\,2\,1^{i}
\end{equation}
and is represented on in Figure \ref{fig:lapartition}.

\begin{figure}[htp]
\centering
\includegraphics{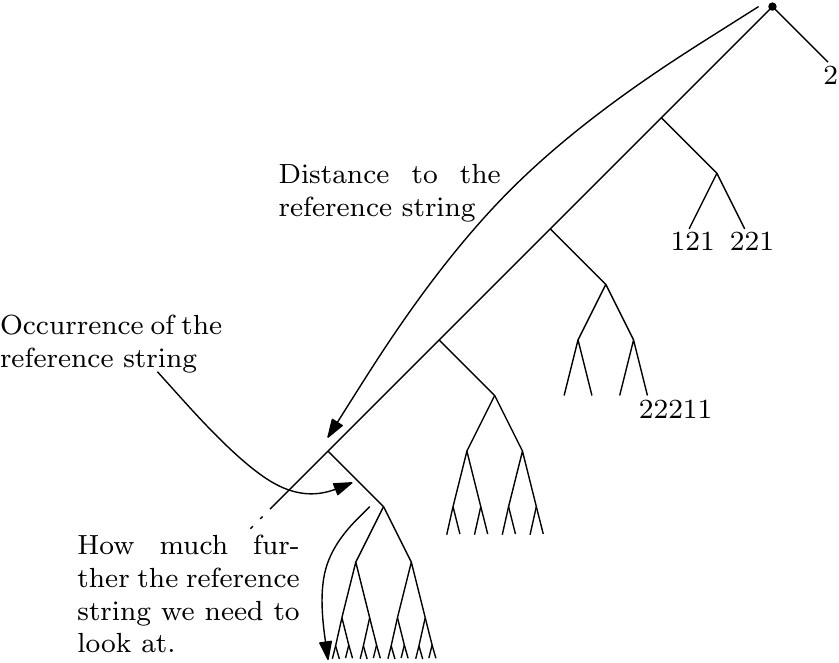}
\caption{The upper part of the context tree $\tau$ defined by (\ref{lexamplecontexttree}). We only specify some of the contexts. }
\label{fig:lapartition}
\end{figure}

Our results say that assuming $\inf_{v\in\tau}p(2|v)\geq\epsilon>0$, we can perfectly simulate the unique stationary chain ${\bf X}$ compatible with $(\tau,p)$ (Theorem \ref{theo1}). The perfect simulation algorithm extends the algorithm we described rapidly in Subsection \ref{renewall} for the renewal chain. The main difference is in the definition of the regeneration time. In Section \ref{realistic}, we will give an explicit perfect simulation of this chain. We also show in this work that almost surely, infinitely many occurrences of $2$  split the realization of ${\bf X}$ into independent and identically distributed strings (Theorem \ref{theo2}). However, since any occurrence of $2$ can be bypassed by a future context, with positive probability, this ``regeneration scheme'' differs substantially of the ``renewal scheme''.

\section{On the form of our context trees}\label{ontheform}

The examples of the preceding Section gave us an idea of how our families can be described in terms of probabilistic context trees. The aim of the following definitions is to define, using the probabilistic context tree framework, what we called ``reference string''.
At the end of this section, we give several examples explaining  these definitions. 

Suppose we are given an unbounded context tree $\tau$. For any finite string $w$ of $A^{\star}$ we define the function $m^{w}$ which associates to any context $v\in\tau$ the integer number
\begin{equation}\label{mw}
m^{w}(v):= \inf\{j: 0 \le j \le |v|-|w|,\, \mbox {\,such that\, \,}
v^{-j-1}_{-j-|w|}=w \}\, ,
\end{equation}
with the convention that $m^{w}(v) =+\infty$ if the set of indexes is
empty. In the context tree, $m^{w}(v)$ is the distance between the root and the first occurrence of $w$ in the context $v$.
If a context $v$ is such that
$m^{w}(v)=k$, then it can be written as the concatenation
\[
v = v_{-|v|}\ldots v_{-k-|w|-1}\,\,w\,\, v_{-k} \ldots v_{-1},
\]
where $v_{-j}^{-j+|w|-1}\neq w$ for $j=|w|,\ldots,k+|w|-1$.
The context trees considered in the present work have the following form. There exist a finite string $w$ of $A^{\star}$ and, related to this string, a function $\ell^{w}:\mathbb{N}\rightarrow\mathbb{N}$ satisfying $\ell^{w}(k)<+\infty$ for any $k\geq0$, such that for any $v\in\tau$
\[
|v|= m^{w}(v)+|w|+\ell^{w}(m^{w}(v)).
\]
The string $w$ is the \emph{reference string} for the context tree $\tau$. The function $\ell^{w}$ tells us ``how much further'' the last occurrence of $w$ we need to look back. It is precisely the notion of reference string which generalizes the notion of renewal string. 

Let us give some examples for the reader  to see how the notion of  reference string appears on the shape of the context trees.

\subsection{Example of Figure \ref{fig:infinite-pct}}

The  reference string is the symbol $2$. The function $\ell^{2}$ is identically $0$,
traducing the fact that $2$ is a renewal symbol.

\subsection{Example of Figures \ref{fig:lapartition} and \ref{fig:partition}}

The  reference string is also the symbol $2$ in both figures. In figure \ref{fig:lapartition}, $\ell^{2}$ is the identity
as we say in Section \ref{anexample}.  Comparing Figures \ref{fig:infinite-pct} and \ref{fig:lapartition} let clear the fact that the notion of  reference string generalizes the notion of renewal symbol. 

To simplify Figure \ref{fig:partition}, we made small triangles for subtrees. These subtrees are context trees of finite height since $2$ is a  reference string.
Suppose the context tree illustrated in Figures \ref{fig:partition} is such that $\ell^{w}(k)=1+k^{2}$. It follows that any context of the context tree of $\tau_{\textrm{Figure \ref{fig:lapartition}}}$ is a suffix of a context of $\tau_{\textrm{Figure \ref{fig:partition}}}$. In this case, we write $\tau_{\textrm{Figure \ref{fig:lapartition}}}\leq\tau_{\textrm{Figure \ref{fig:partition}}}$, and observe that
\[
\tau_{\textrm{Figure \ref{fig:infinite-pct}}}\leq\tau_{\textrm{Figure \ref{fig:lapartition}}}\leq\tau_{\textrm{Figure \ref{fig:partition}}}
\]
There is a difference between $\tau_{\textrm{Figure \ref{fig:partition}}}$ and the two others: the  reference string $2$ \emph{is not} a context. 


\subsection{Example of Figure \ref{fig:partition3}}

The  reference string is  $12$ and is not a context. We can see $5$ infinite size contexts but there is infinitely many of them. 

\begin{figure}[h!]
\begin{center}
\subfigure[]{
\includegraphics{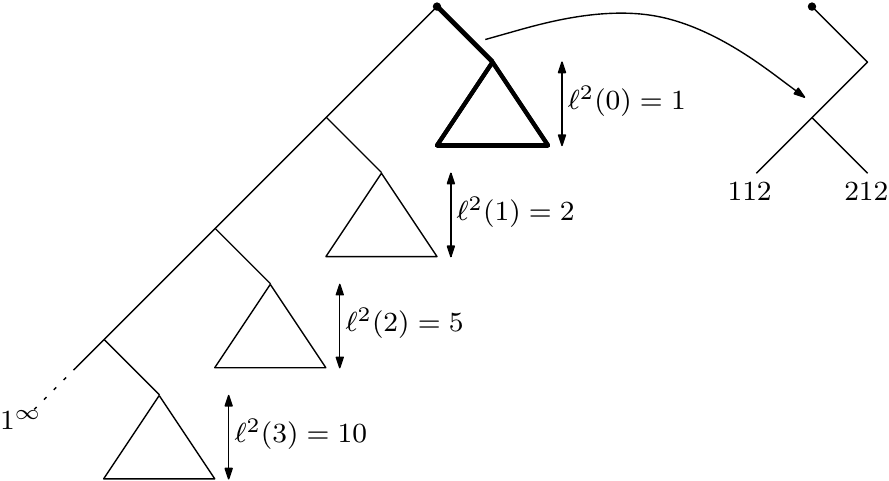}
\label{fig:partition}
}
\hspace{1cm}
\subfigure[]{
\includegraphics{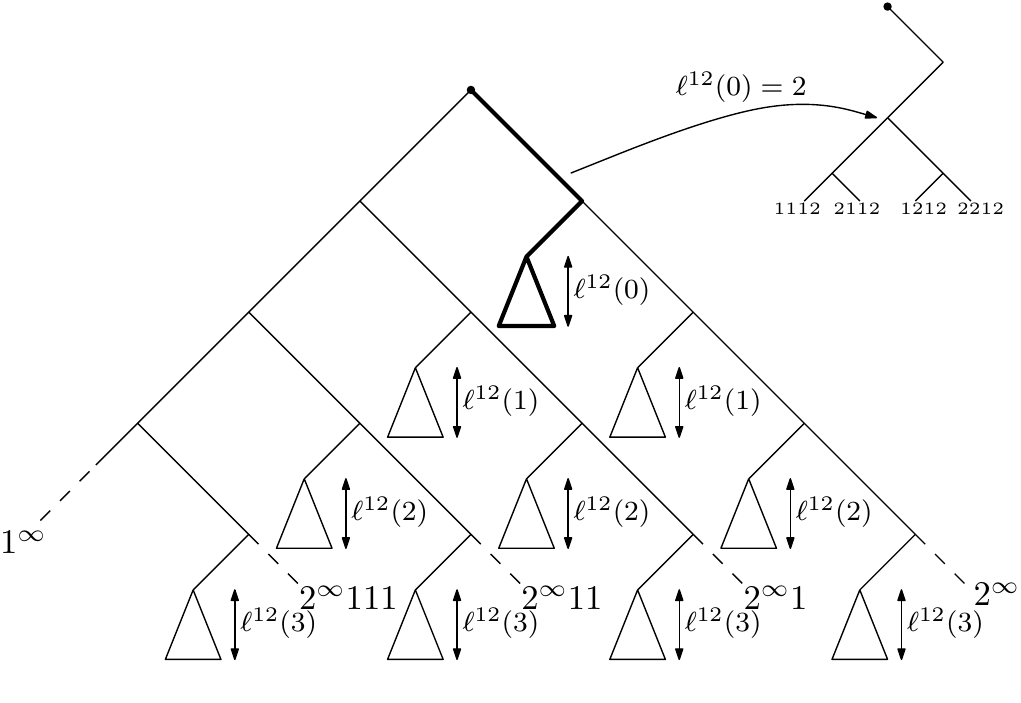}
\label{fig:partition3}
}
\caption{The upper part of two unbounded context trees, the one of (a) has as reference string the symbol $2$, the one of (b) has the string $12$.}
\end{center}
\end{figure}


\section{Perfect simulation algorithm and Statement of the results}\label{algorithm}

Consider an unbounded probabilistic context tree $(\tau, p)$. We recall that $\mathcal{E}$ is the set of elements of $A$ which are $\epsilon$-regulars for some $\epsilon>0$. 
We will assume without loss of generality that $A=\{1,2,\ldots\}$ and $\mathcal{E}=\{1,2,\ldots,\#\mathcal{E}\}$ ($\#\mathcal{E}$ denotes the cardinality of $\mathcal{E}$). 
We introduce a partition of $[0,1[$, illustrated in Figure  \ref{fig:partitionof[0,1]}, which will be used for the \emph{construction} of the chain.
Define for any $a\in \mathcal{E}$ and any $v\in\tau$ the intervals
\[
J(a|\emptyset)= [(a-1)\epsilon, a\epsilon[\,\,,\,\,\,\,\,\,J(a|v)=
\left[\#\mathcal{E}\epsilon+\sum_{i=1}^{a-1}(p(i|v)-\epsilon),\#\mathcal{E}\epsilon+\sum_{i=1}^{a}(p(i|v)-\epsilon)\right[
\]
and 
\[
K(a|v)= J(a|\emptyset)\textrm{ }\cup J(a|v).
\]
For any $a\in A\setminus\mathcal{E}=\{\#\mathcal{E}+1,\#\mathcal{E}+2,\ldots\}$
\[
J(a|v)=K(a|v)=
\left[\sum_{i=1}^{a-1}p(i|v),\sum_{i=1}^{a}p(i|v)\right[.
\]
Observe that for any $v\in\tau$
\[
J(1|\emptyset),\ldots,J(\#\mathcal{E}|\emptyset),J(1|v),\ldots,J(\#\mathcal{E}|v),J(\#\mathcal{E}+1|v),\ldots
\]
defines a partition of $[0,1[$ (see Figure \ref{fig:partitionof[0,1]}), and that for any $a\in A$ and any $v\in\tau$
\begin{equation}\label{lebesgue}
\lambda(K(a|v))=p(a\vert v)
\end{equation}
where $\lambda$ denotes the Lebesgue measure on $[0,1[$.

\vspace{0,3cm}
Let ${\bf U}=(U_{n})_{n \in \Z}$ be a sequence of i.i.d. random variables
uniformly distributed in $[0,1[$ and defined in some probability space
$(\Omega,\mathcal{F},\mathbb{P})$. All the chains considered in
what follows will be constructed using this sequence ${\bf U}$.

\begin{figure}[htp]
\centering
\includegraphics{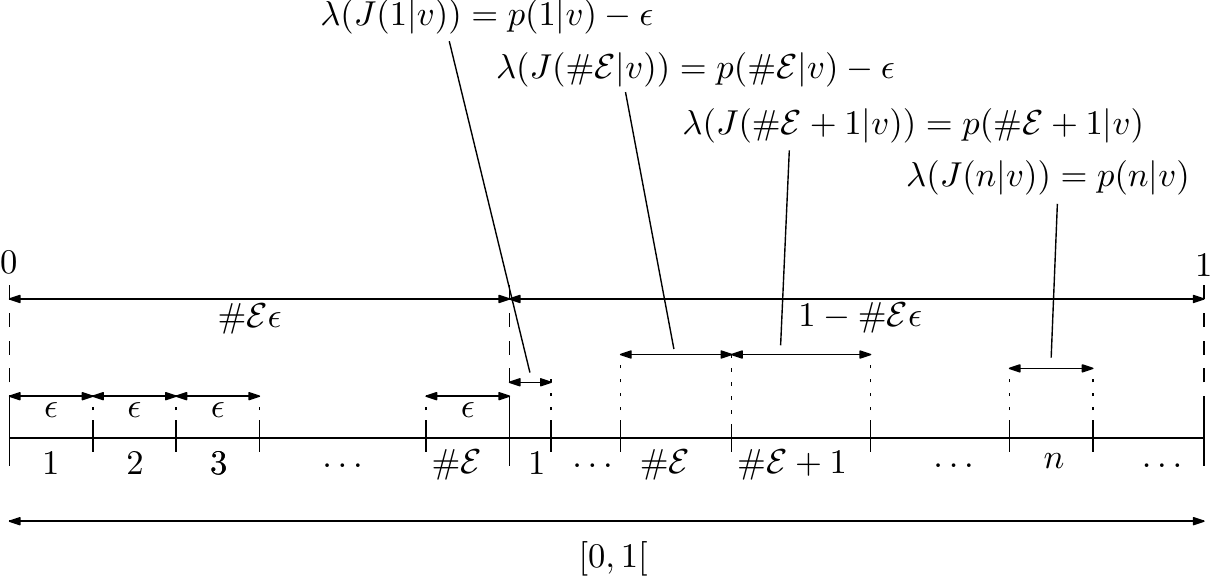}
\caption{Illustration of the partition of $[0,1[$ with the disjoint intervals $\{J(a|\emptyset)\}_{a\in \mathcal{E}}$ and $\{J(a|v)\}_{a\in \mathcal{A}}$ for some $v\in\tau$.}
\label{fig:partitionof[0,1]}
\end{figure}

\vspace{0,3cm}
We  construct a deterministic measurable function $X:[0,1[^{\mathbb{Z}}\rightarrow A^{\mathbb{Z}}$ such that the law $\mathbb{P}(\xU\in \cdot )$  is compatible with $(\tau,p)$. 
The construction of this function is carried out in such a way that for any $n\in\mathbb{Z}$, $\XU_{n}=a$ whenever $U_{n}\in J(a|\emptyset)$, for any $a\in \mathcal{E}$. 
Suppose that for some time index  $n\in\mathbb{Z}$ there exists a string $a_{-k}^{-1}\in A^{k}$ such that $U_{n-i}\in J(a_{-i}|\emptyset),\,i=1,\ldots,k$, in this case, we put
\[
\XU^{n-1}_{n-k}=a_{-k}^{-1}.
\] 
This is a sample that has been \emph{spontaneously} constructed. We have three situations at this point: (i) $U_{n}$ belongs to $[0,\#\mathcal{E}\epsilon[$, (ii) $U_{n}$ belongs to $[\#\mathcal{E}\epsilon,1[$ and $c_{\tau}(a_{-k}^{-1})=v\in\tau$, and (iii) $U_{n}$ belongs to $[\#\mathcal{E}\epsilon,1[$ and $c_{\tau}(a_{-k}^{-1})=\emptyset$. In situation (iii), we are not able to construct $\XU_{n}$ knowing \emph{only} $\XU_{n-k}^{n-1}$, and therefore, we need to determine more past symbols. In situations (i) and (ii), we can construct $\XU_{n}$ independently of $\XU_{-\infty}^{n-k-1}$: we put for any $a\in A$
\[
\XU_{n}=a \mbox{\, \, if\, \, } U_{n} \in K(a|v).
\]
Observe that $\XU^{n}_{n-k}$ has been constructed independently of $U_{-\infty}^{n-k-1}$ and $U_{n+1}^{+\infty}$. 
Suppose we want to sample the value of the stationary chain at time $0$. The idea of the algorithm is to generate the $U_{i}$'s backward in time until the first time $k\leq0$ in the past such that we can carry out the above  construction from $k$ to $0$, without using  $U_{-\infty}^{k-1}$ and $U_{1}^{+\infty}$. This is a CFTP algorithm. 

The size of the suffix of the past we need to know  to construct the next symbol depends here on the previously constructed past itself (excepted for the time indexes where the symbol is spontaneously constructed). In the CFTP algorithm introduced by \cite{comets/fernandez/ferrari/2002}, the size of this suffix of the past is defined by an i.i.d. random variable, totally independent of values assumed by the chain. This is the main difference between both works on the technical point of view, and this makes our perfect simulation algorithm a little bit more complicated. 
\vspace{0,1cm}
The three cases enumerated above by (i), (ii) and (iii) are formally described by the measurable function $F:[0,1[\times(A^{\star}\cup A^{-\mathbb{N}})\rightarrow A\cup\{\star\}$ defined as follows: for any $a_{m}^{n}\in A^{\star}\cup A^{-\mathbb{N}}$, $-\infty\leq m\leq n+1$, 
\begin{equation}\label{F}
F(u,a_{m}^{n})=\sum_{a\in A}a.\mathbbm{1}\{u\in K(a|c_{\tau}(a_{m}^{n}))\}+\star\mathbbm{1}\{u\in [\#\mathcal{E}\epsilon,1[,\, c_{\tau}(a_{m}^{n})=\emptyset\},
\end{equation}
with the conventions that $a_{n+1}^{n}=\emptyset$  and $c_{\tau}(\emptyset)=\emptyset$.
When  $c_{\tau}(a_{m}^{n})\neq \emptyset$, we have
\begin{equation}\label{lebe}
\mathbb{P}(F(U_{n+1},a_{m}^{n})=a)=\mathbb{P}(U_{n+1}\in K(a|c_{\tau}(a_{m}^{n})))=\lambda(K(a|c_{\tau}(a_{m}^{n})))=p(a|c_{\tau}(a_{m}^{n})).
\end{equation}
This function is an \emph{update function}. Since we consider chains of infinite order, this update function may return the symbol $\star$, meaning that we have not a sufficient knowledge of the past to continue the construction. 

We define for any $m\leq n$ the $\mathcal{F}(U_{m}^{n})$-measurable function $\mathcal{L}: [0,1[^{n-m+1}\rightarrow \{0,1\}$ which takes value $1$ if and only if we can construct $\XU_{m}^{n}$ independently of $U_{-\infty}^{m-1}$ and $U_{n+1}^{+\infty}$ using the construction described above. Formally
\[
\{\mathcal{L}(U_{m}^{n})=1\}:=\bigcup_{a_{m}^{n}\in A^{n-m+1}}\bigcap_{i=m}^{n}\{F(U_{i},a_{m}^{i-1})=a_{i}\}.
\]
Finally, we define for any $m\leq n\leq+\infty$
\begin{equation}\label{formaldef}
\theta[m,n]:=\max\{k\leq m:\mathcal{L}(U_{k}^{n})=1\}
\end{equation}
with the convention that  $\theta[m]:=\theta[m,m]$. This time is called \emph{regeneration time} for the window $[m,n]$, and is the first time before $m$ such that the construction described above is successful until time $n$.

We now state our main results and present a ``simplistic'' perfect simulation algorithm (Algorithm 1) for the construction of a sample $\XU_{m}^{n}$. A more ``realistic'' one (Algorithm 2) is given in Section \ref{realistic} together with an explicit perfect simulation in the particular case of Section \ref{anexample}.

\begin{algorithm}[h]
\caption{``Simplistic'' perfect simulation algorithm of the sample $\XU_{m}^{n}$} 
\begin{algorithmic}[1]
\STATE {\it Input:} $m$, $n$; {\it Output:} $\theta[m,n]$, $(\XU_{\theta[m,n]},\ldots,\XU_{n})$
\STATE  Sample $U_{m},\ldots,U_{n}$ uniformly in $[0,1[$\\
\vspace{0.1cm}
\STATE $i \leftarrow m$, $\theta[m,n] \leftarrow m$, $\XU_{m}^{n}\leftarrow\star^{n-m+1}$, $\mathcal{L}(U_{m}^{n})\leftarrow0$
\WHILE{$\mathcal{L}(U_{i}^{n})\neq 1$}
\STATE $i \leftarrow i-1$\\
Choose $U_{i}$ uniformly in $[0,1[$\\
\ENDWHILE
\STATE $\theta[m,n]\leftarrow i$
\WHILE{$\XU_{n}=\star$}
\STATE $\XU_{i}\leftarrow F(U_{i},\XU_{\theta[m,n]}^{i-1})$\\
$i \leftarrow i+1$
\ENDWHILE

\RETURN $\theta[m,n]$, $(\XU_{\theta[m,n]},\ldots, \XU_{n})$

\end{algorithmic}
\end{algorithm}

\begin{theo}\label{theo1}\emph{(Perfect simulation).}
Consider a probabilistic context tree $(\tau,p)$ having an $\epsilon$-regular  reference string $w$. If
\begin{equation}\label{rate-increasing}
\limsup_{k\rightarrow\infty}\frac{\log(\ell^{w}(k))}{C_{\epsilon}k}<1\,\,\,\,,\,\,\,\,\,\,\,C_{\epsilon}:=-\frac{1}{|w|}\log(1-\epsilon^{|w|})>0
\end{equation}
then Algorithms 1 and 2 stop almost surely after a finite number of steps, i.e., we have for any $-\infty<m\leq n\leq+\infty$ 
\begin{equation}\label{a.s.}
\mathbb{P}( \theta[m,n] >- \infty ) = 1.
\end{equation}
\end{theo}

In the rest of the paper, we will often write $X_{i}$ for $\XU_{i}$ (and ${\bf X}$ for $\xU$) in order to avoid overloaded notations, keeping in mind the fact that for any $i$, $X_{i}$ is constructed as a deterministic function of ${\bf U}$. Actually, by Theorem \ref{theo1},  $X_{i}$ depends only on a $\mathbb{P}$-a.s. finite part of this sequence:  $X_{i}:=[X(\ldots,u_{\theta[i]-1},U_{\theta[i]},\ldots,U_{i},u_{i+1},\ldots)]_{i}$ for any ${\bf u}\in [0,1[^{\mathbb{Z}}$. 

\begin{coro}\label{coro1}\emph{(Existence and uniqueness)}.
The output of Algorithms 1 and 2 are samples of the unique stationary chain compatible with $(\tau,p)$. We will call $\mu$ the stationary measure of ${\bf X}$:
\[
\mu:=\mathbb{P}(X({\bf U})\in \cdot).
\]
\end{coro}

Note that the $\epsilon$-regularity assumption for the symbols appearing in the context $w$ is weaker than the \emph{regularity} (also called \emph{strongly non-nullness}) assumption of the literature. And we know that this later condition neither implies existence nor uniqueness of the stationary measure \citep[see for example][]{bramson/kalikow/1993}. Our $\epsilon$-regularity conditions may rather be  compared to the weakly non-nulness assumption which requires $\sum_{a\in A}\inf_{v\in\tau}p(a|v)>0$ and which is  assumed for example by \cite{comets/fernandez/ferrari/2002}. In fact, this condition is very useful for our construction point of view: this allows us to have symbols which appear spontaneously, and makes the CFTP easier to perform.

The proof of Corollary \ref{coro1} using the CFTP algorithm and Theorem \ref{theo1} can be found in  \citet{comets/fernandez/ferrari/2002}  (Proposition 6.1 for the existence statement and Corollary 4.1 for the uniqueness statement). We omit these proofs in the present work in order to save space, but we mention the main lines. The existence statement follows once we observe that Theorem \ref{theo1} implies that one can construct a bi-infinite sequence ${\bf X}$ verifying for any $n\in \mathbb{Z}$, $X_{n}=F(U_{n},X_{-\infty}^{n-1})$. By (\ref{lebe}), this chain is therefore compatible in the sense of (\ref{compatible}). It is stationary by construction. The uniqueness statement follows from the loss of memory the chain inherits because of the existence of almost surely finite regeneration times.

\vspace{0.1cm}
We call time $t$ a regeneration  time for the chain ${\bf X}$ if  $\theta[t,+\infty]=t$.
Define the chain $\boldsymbol{\xi}$ on $\{0,1\}$ by $\xi_{j}:=\mathbbm{1}\G\{j=\theta[j,+\infty]\D\}$.
Then, consider the sequence of time indexes ${\bf T}$ defined such that $\xi_{j}=1$ if and only if $j=T_{l}$ for some $l$ in $\mathbb{Z}$, $T_{l}< T_{l+1}$ and with the convention $T_{0}\leq0<T_{1}$. We say that ${\bf X}$ has a regeneration scheme if the chain $\boldsymbol{\xi}$ is renewal (that is, if the  increments $(T_{i+1}-T_{i})_{i\in\mathbb{Z}}$ are independent, and are identically distributed for $i\neq0$).

\begin{coro}\label{coro2}\emph{(Regeneration scheme)}.
In the conditions of Theorem \ref{theo1} the chain ${\bf X}$ has a regeneration scheme. The random strings $(\XU_{T_{i}},\ldots,\XU_{T_{i+1}-1})_{i\neq0}$ are i.i.d. and have finite expected size.
\end{coro}

In words, this corollary states that the unique stationary chain compatible with $(\tau,p)$ under the  conditions of Theorem \ref{theo1} can be viewed as an i.i.d. concatenation of strings of symbols of $A$ having finite expected size. A similar result has been first obtained by \cite{lalley/1986} for one dimensional Gibbs states under appropriate conditions on the continuity rate, and then by \cite{comets/fernandez/ferrari/2002} under weaker conditions than the ones of \cite{lalley/1986}. 
It is an hidden regeneration scheme, because it uses the sequence ${\bf U}$. The main reason why we give this result is that it arises naturally from our perfect simulation approach. 

The visible regeneration scheme involves several technical complications, even if in spirit, it is similar to the preceding one. We postpone the precise definitions to Section \ref{proofoftheorem2}, and give the following simplified statement.

\begin{theo}\label{theo2}\emph{(Visible regeneration scheme)}.
Suppose $(\tau,p)$ satisfies the conditions of Theorem \ref{theo1}. Then, for $\mu$-a.s. realization of the chain ${\bf X}$ compatible with $(\tau,p)$ there exists a sequence of random times ${\bf T}^{{\bf X}}$ such that
\begin{itemize}
\item for any $i\in\mathbb{Z}$, the event $\{T_{i}^{{\bf X}}=k\}$ is measurable with respect to the $\sigma$-algebra generated by $X_{k}^{+\infty}$ and
\item conditionally on ${\bf T}$, the strings $(X_{T^{{\bf X}}_{i}},\ldots,X_{T^{{\bf X}}_{i+1}-1})_{i\neq0}$
are i.i.d. and have finite expected size.
\end{itemize}
\end{theo}

\begin{obs}[Monotonicity]\label{mono}
Suppose  the probabilistic context tree $(\tau,p)$ satisfies the conditions of the above results, then, all the above results hold true for any probabilistic context tree $(\tau',p')$ such that  $\tau'\leq\tau$ and for which $w$ is $\epsilon$-regular.
\end{obs}

\section{Proof of Theorem \ref{theo1}}\label{proof of main theorem}

A slight complication arises from the fact that the random variable $\theta$ depends on the values assumed by the chain ${\bf X}$ along its construction. In the first step of the proof, Subsections \ref{simpli}, \ref{exemplelegal}, \ref{simplicontinued} and \ref{exemplecool}, we define another random variable, we will denote $\bar{\theta}$ (see (\ref{bartheta})) and which has the following properties: (i) it only depends on the spontaneous occurrences of $w$   along the construction, (ii) it can be used to define a lower bound for $\theta$. 

In a  second step, Subsections \ref{2step} and \ref{2step2}, we relate the distribution of $\bar{\theta}$ with the probability of return to the state $0$ for an $\mathbb{N}$-valued auxiliary process which also depends on the spontaneous occurrences of $w$. At this point, there is a clear similarity with the proof of \cite{comets/fernandez/ferrari/2002}, the principal difference being that our auxiliary process is not the house of card process, but is defined through (\ref{key}). 

The proof of Theorem \ref{theo1} is finally given in Subsection \ref{laprueba}.


\subsection{Simplification of the problem}\label{simpli}

Suppose we are given a probabilistic context tree $(\tau,p)$ having an $\epsilon$-regular reference string $w=w_{-|w|}^{-1}$ to which corresponds the function $\ell^{w}$. Owning to Observation \ref{mono}, there is no  loss of generality in restricting the  proof to the case where (i) only the branches having $w$ as subsequence have finite length and (ii) $\ell^{w}$ increases and goes to infinity. Observe that this is for example the case of the context tree illustrated in Figure \ref{fig:partition3}.

We define a new stochastic chain ${\bf Z}$:  for any $i\in\mathbb{Z}$, $Z_{i}=a$  if $U_{i}$ belongs to $J(a|\emptyset)$, and  $Z_{i}=\star$ otherwise. This chain takes in account only the symbols which appear spontaneously in ${\bf X}$: $X_{i}=a$ whenever $Z_{i}=a$, and in particular $X_{i-|w|+1}^{i}=w$ whenever $Z_{i-|w|+1}^{i}=w$, for any $i\in\mathbb{Z}$. We also define the $\mathbb{N}$-valued random variables $m_{i}({\bf U})=m_{i}$ and $L_{i}({\bf U})=L_{i}$ as follows:
\[
m_{i}=\inf\G\{k\geq0:Z_{i-k-|w|}^{i-k-1}=w\D\}
\]
which is the distance to the last occurrence of $w$ in $Z_{-\infty}^{i-1}$, and 
\begin{equation}\label{k}
L_{i}=\left\{
\begin{array}{ccc}
0&\textrm{ if }&U_{i}\in [0,\#\mathcal{E}[,\\
m_{i}+|w|+\ell^{w}(m_{i}) &\textrm{ otherwise}.
\end{array}\right.
\end{equation}
The reason why we introduced these random variable is that if we have $\mathcal{L}(U_{m}^{n})=1$ for $-\infty<m<n<+\infty$ and $L_{n+1}\leq n-m+1$, then $L_{n+1}$  is an upper bound for the number of sites in the past we need to know in order to  decide the state at time $n+1$ using the perfect simulation algorithm. 
To see that, suppose that for some $-\infty<m<n<+\infty$ we have $\mathcal{L}(U_{m}^{n})=1$ and that $0<L_{n+1}\leq n-m+1$. Since 
\[
Z_{i-|w|}^{i-1}=w\Rightarrow X_{i-|w|}^{i-1}=w\,,\,\,\forall i\in\mathbb{Z}
\]
it follows that the distance to the last occurrence of $w$ in $Z_{m}^{n}$ is larger than in $X_{m}^{n}$ (recall the definition (\ref{mw}) of $m^{w}$):
\[
m_{n+1}\geq m^{w}(c_{\tau}(X_{m}^{n})).
\]
We also recall that for any $v\in\tau$  in which the string $w$ appears, we have
\begin{equation}\label{contextsize}
|v|=m^{w}(v)+|w|+\ell^{w}(m^{w}(v)),
\end{equation}
therefore, by definition (\ref{k}), whenever  $L_{n+1}>0$, $L_{n+1}$ is an upper bound for the size of the context needed at time $n+1$:
\[
|c_{\tau}(X_{m}^{n})|\leq L_{n+1}.
\]
Observe also that  $L_{n+1}=0$ if and only if the symbol appears spontaneously at time $n+1$.

\subsection{Example of Figure \ref{fig:lapartition}}\label{exemplelegal}
In this context tree, the symbol $2$ is the reference string and suppose that it is the only $\epsilon$-regular symbol. In this particular case, $|w|=1$ and  ${\bf Z}$ turns out to be an i.i.d. chain taking value $2$ with probability $\epsilon$, and $\star$ with probability $1-\epsilon$. Let us consider the random variable
\begin{equation}\label{bartheta1}
\theta^{|w|=1}[0,n]:=\max\big\{j\leq 0:L_{i}\leq i-j,\,\,i=j,\ldots, n\big\}.
\end{equation}
The best way to understand the utility of this new random variable is to exlpain that 
for any $n\geq0$, when we are in the set $\{{\bf U}:\theta^{|w|=1}[0,n]>-\infty\}$ we have
\begin{equation}\label{w1}
\theta^{|w|=1}[0,n]\leq \theta[0,n].
\end{equation}
To simplify the notation let us  write $\theta^{1}:=\theta^{|w|=1}[0,n]$. To each time $i\in\{\theta^{1},\ldots,n\}$, we associate an arrow going from time $i$ to time $i-L_{i}$. The definition of $\theta^{1}$ says that no arrows starting from $\{\theta^{1},\ldots,n\}$ go beyond time $\theta^{1}$. This means that the construction of $X_{\theta^{1}}^{n}$ can be performed recursively from time $\theta^{1}$ to time $n$ using only $U_{\theta^{1}}^{n}$, and therefore that 
%
$\mathcal{L}(U_{\theta^{1}}^{n})=1$. Since $\theta[0,n]$ is the maximum over $\{k\leq0:\mathcal{L}(U_{k}^{n})=1\}$, it follows that  (\ref{w1}) holds.

\subsection{Simplification of the problem (continued)}\label{simplicontinued}

To keep the same form as $\theta^{|w|=1}[0,n]$ for the general case where $|w|\geq1$, we introduce the time rescaled 
chain $\bar{{\bf Z}}$ defined by
\begin{equation}\label{barz}
\bar{Z}_{m}=\left\{
\begin{array}{ccc}
1&\textrm{ if }&U_{m|w|-i+1}\in J(w_{-i}|\emptyset),\,\,\,i=0,\ldots,|w|-1\\
\star &\textrm{ otherwise},
\end{array}\right.
\end{equation} 
and the function
\begin{equation}\label{barlbarh}
\bar{\ell}(i):=\G\lceil\frac{\ell^{w}\G((i+1)|w|-1\D)}{|w|}\D\rceil,
\end{equation}
where for any $r\in \mathbb{R}$, $\lceil r\rceil$ denotes the smallest integer number greater than or equal to $r$. Using these new definitions, we introduce the rescaled random variables
\[
\bar{m}_{i}=\inf\G\{k\geq0:\bar{Z}_{i-k-1}=1\D\}
\]
which is the distance to the last occurrence of $1$ in $\bar{Z}_{-\infty}^{i-1}$ and 
\begin{equation}\label{bark}
\bar{L}_{i}=\left\{
\begin{array}{ccc}
0&\textrm{ if }&\bar{Z}_{i}=1,\\
\bar{m}_{i}+1+\bar{\ell}(\bar{m}_{i}) &\textrm{ otherwise}.
\end{array}\right.
\end{equation}
We are now able to define our new random time:
\begin{equation}\label{bartheta}
\bar{\theta}[0,n]:=\max\big\{j\leq 0:\bar{L}_{i}\leq i-j,\,\,i=j,\ldots, n\big\}.
\end{equation}
Observe that in the case where $|w|=1$, this definition is equivalent to definition (\ref{bartheta1}).

\begin{figure}[htp]
\centering
\includegraphics{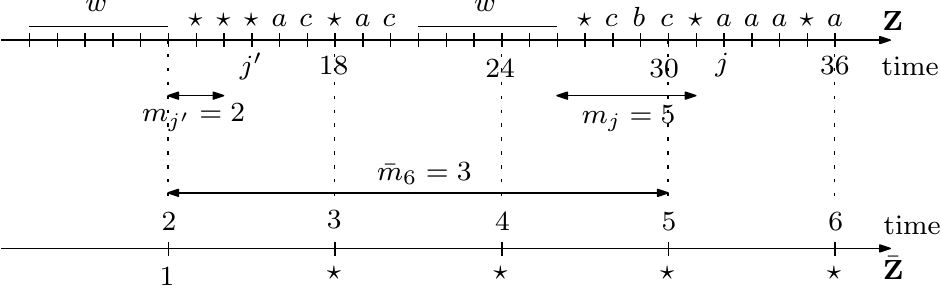}
\caption{Illustration of two coupled samples $Z_{12}^{36}$ and $\bar{Z}_{1}^{6}$ in a case where the reference string $w$ has size $6$, and $\mathcal{E}=\{a,b,c\}$. We have $m_{j}=5$, $m_{j'}=2$ and $\bar{m}_{6}=3$ and  $\bar{m}_{3}=0$. This Figure  illustrates the  inequalities (\ref{in1}) and (\ref{in2}).}
\label{fig:illustr}
\end{figure}

\begin{lemma}\label{matador}
For any $n\geq0$, we have in the set $\{{\bf U}:\bar{\theta}[0,n]>-\infty\}$ 
\[
\theta[0,n|w|]\geq|w|(\bar{\theta}[0,n]-1)+1.
\]
\end{lemma}

\noindent We refer to Subsection \ref{exemplecool} for an example illustrating this Lemma.

\begin{proof}
The way we defined $m_{i}$ and $\bar{m}_{i}$ implies that  for $j\in\{(i-1)|w|+1,\ldots,i|w|\}$,
\begin{equation}\label{in1}
m_{j}\leq (\bar{m}_{i}+1)|w|-1-(i|w|-j)
\end{equation}
\begin{equation}\label{in2}
\leq (\bar{m}_{i}+1)|w|-1.
\end{equation}
We refer to Figure \ref{fig:illustr} for a pictorial illustration of these inequalities.
It follows that, whenever $\bar{L}_{i}>0$, we have the following sequence of inequalities: for any $j$ in the set of sites $\{(i-1)|w|+1,\ldots,i|w|\}$ (recall that $\ell^{w}$ is increasing)
\begin{equation}\label{uuu}
\begin{array}{ccccc}
L_{j}&:=&m_{j}+|w|+\ell^{w}(m_{j})
\\
\vspace{0,1cm}
&\leq&m_{j}+|w|+\ell^{w}((\bar{m}_{i}+1)|w|-1)&\textrm{by inequality(\ref{in2})}
\\
\vspace{0,1cm}
&\leq& m_{j}+|w|+|w|\bar{\ell}(\bar{m}_{i})&\textrm{by definition (\ref{barlbarh})}
\\
\vspace{0,1cm}
&\leq&|w|-1-(i|w|-j)+|w|\bar{L}_{i}&\textrm{by inequality(\ref{in1})}.
\end{array}
\end{equation}
This indicates that, whenever $\bar{L}_{i}>0$, the arrow starting at time $j$ goes, at most, until time $|w|-1-(i|w|-j)+|w|\bar{L}_{i}$. In the case where $\bar{L}_{i}=0$, we have $L_{j}=0$ for $j\in\{(i-1)|w|+1,\ldots,i|w|\}$.
Denote $\bar{\theta}=\bar{\theta}[0,n]$,
the last line of  (\ref{uuu}) yields the following inclusion
\[
\bigcap_{i=\bar{\theta}}^{n}\{\bar{L}_{i}\leq i-\bar\theta\}\subset\bigcap_{i=\bar\theta}^{n} \bigcap_{j=(i-1)|w|+1}^{i|w|}\{L_{j}\leq j-|w|(\bar\theta-1)-1\},
\]
meaning that in the sequence $\bar{\bf Z}$, none of the arrows starting from the set of sites $\{|w|(\bar{\theta}-1)+1,\ldots,n|w|\}$ will pass time $|w|(\bar{\theta}-1)+1$. 

\end{proof}

\subsection{Example illustrating Lemma \ref{matador}}\label{exemplecool}

The lower part of Figure \ref{fig:combine} illustrates Lemma   \ref{matador} in a case where our reference string $w$ has size $|w|=3$, and $\mathcal{E}=\{a,b,c\}$. We represented two coupled samples $Z_{-38}^{6}$ and  $\bar{Z}_{-12}^{2}$. 
The subjacent sample $U_{-38}^{6}$ with which both sequences are coupled is not represented here. We use the function $\ell^{w}$ having the following values:
\[
\ell^{w}(0)=0,\,\ell^{w}(1)=0,\, \ell^{w}(2)=2,\,\ell^{w}(3)=2,\,\ell^{w}(4)=3,
\]
\[
\ell^{w}(5)=4,\,\ell^{w}(6)=7,\,\ell^{w}(7)=8,\,\ell^{w}(8)=12.
\]
The function   $\bar{\ell}$ has the following values, calculated with $\bar{\ell}(i):=\G\lceil\frac{\ell^{w}\G((i+1)|w|-1\D)}{|w|}\D\rceil$:
\begin{equation}\label{valuesbarl}
\bar{\ell}(0) =1,\,\bar{\ell}(1)= 2,\,\bar{\ell}(2)=4.
\end{equation}
The arrows starting at time $i$ represent $L_{i}$ in the upper sequence, and $\bar{L}_{i}$ in the lower one. The rectangle delimits the sample we want to construct. 
%
This example shows that we can use the sequence $\bar{\bf Z}$ to define the lower  bound $(\bar{\theta}[0,2]-1)|w|+1=38$ for $\theta[0,6]$. However, this lower bound is quite large, since we observe that the perfect simulation can be done from time $-5$ instead of time $-38$!
Even more, it is possible that the perfect simulation can be performed from time $-1$, but we cannot check this here, because we need the information of the context tree and of $U_{-1}^{6}$.
\begin{figure}[htp]
\centering
\includegraphics[scale=0.9]{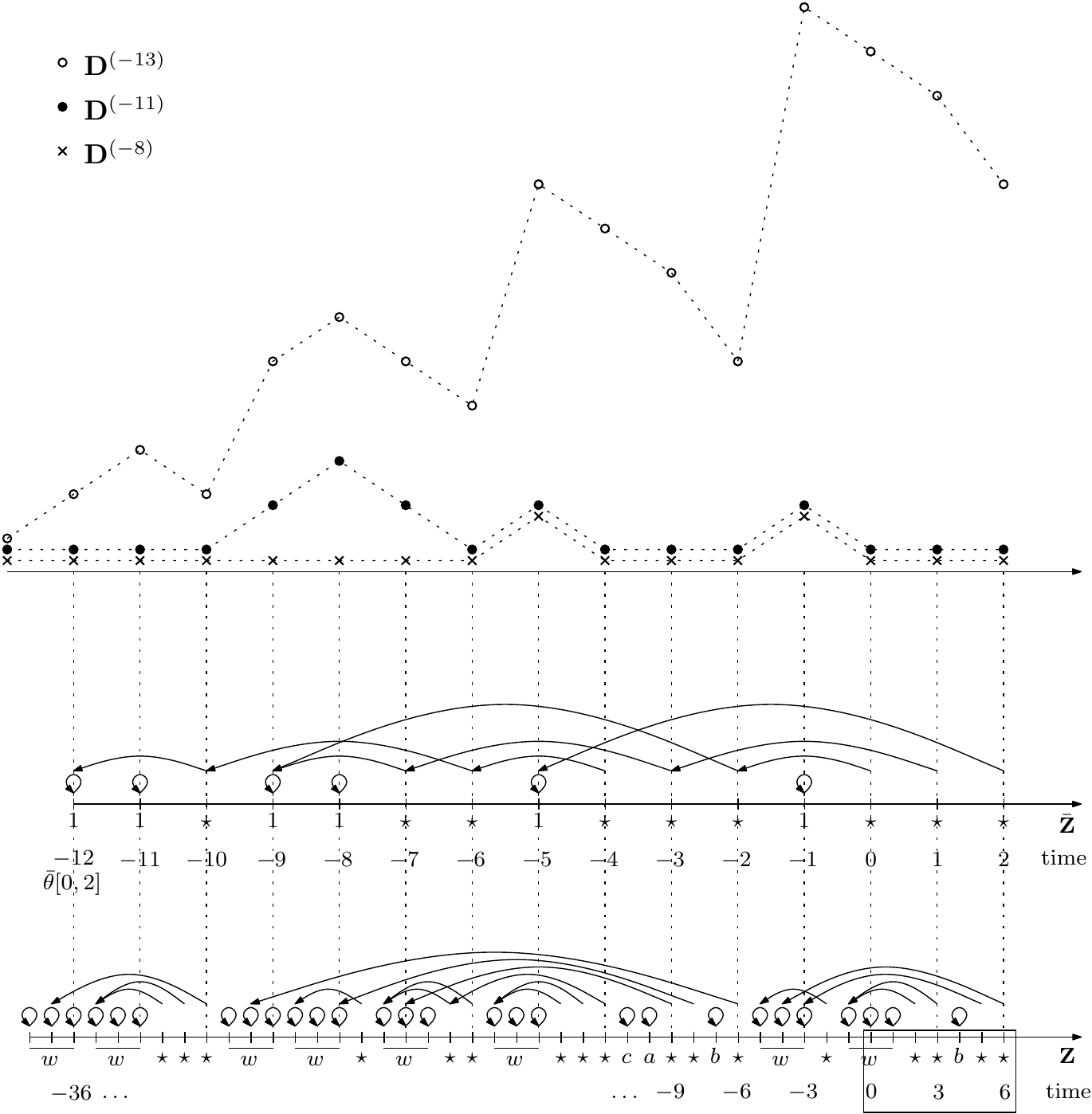}
\caption{Illustration of the random variable  $\bar{\theta}[0,2]$ constructed using $\bar{{\bf Z}}$, and of the behavior of the chains ${\bf D}^{(i)}$ for $i=-13,-11,-8$.}
\label{fig:combine}
\end{figure}


\subsection{An auxiliary process to study $\bar{\theta}[0,n]$}\label{2step}

For any $n\in\mathbb{Z}$, let us consider the $\mathbb{N}$-valued stochastic chain ${\bf D}^{(n)}$ defined by $D_{i}^{(n)}=0$ for any $i\leq n$ and
\begin{equation}\label{key}
D_{i}^{(n)}=(i-i^{(n)}-\bar{L}_{i})\vee 0\,\,,\,\,\,\forall i\geq n+1,
\end{equation}
where $i^{(n)}:=\max\{l< i:D_{l}^{(n)}=0\}$. 
We refer to Figure \ref{fig:combine} for an illustration of this process in the case of the example of Subsection \ref{exemplecool}. In this figure, we illustrated the samples $D^{(i)}_{-13},\ldots,D^{(i)}_{2}$ for $i=-8,-11,-13$, related to the sample $\bar{Z}_{-12}^{2}$.


As we informally mentioned in the introduction of the present Section, we introduce this new auxiliary stochastic chain in order to study the distribution of $\bar{\theta}$. The relation between $\bar{\theta}[0,n]$ and ${\bf D}^{(0)}$ is made clear by the next  Lemma.
 Let us mention that ${\bf D}^{(n)}$ has the same role as the house of cards process ${\bf W}^{n}$ introduced in  \citet[Section 5]{comets/fernandez/ferrari/2002}.

\begin{lemma}\label{seqgeq1}For any $l\geq0$ and $n\geq0$
\[
\Prob(\bar{\theta}[0,n]<-l)\leq\sum_{k= l+1}^{l+n+1}\Prob(D_{k}^{(0)}=0).
\]
\end{lemma}

\begin{proof}
It follows directly from the above definition and definition (\ref{bartheta}) of $\bar{\theta}$ that
\begin{equation}\label{stat1}
\G\{\bar{\theta}[0,n]<-l\D\}=\bigcap_{i=-l-1}^{-1}\bigcup_{k=i+1}^{n}\G\{D_{k}^{(i)}=0\D\}.
\end{equation}
We also observe that for any two integers $n_{1}\leq n_{2}$ 
\begin{equation}\label{domiprop1}
D^{(n_{1})}_{i}\geq D^{(n_{2})}_{i},\,\,\,\forall i\in\Z.
\end{equation}
Moreover, if  $D_{i}^{(n)}=0$ for some  $i>n$,  then $D_{j}^{(n)}=D^{(i)}_{j}$ for all $j\geq i$. These two facts imply that the sequence of chains  $({\bf D}^{(n)})_{n\in\mathbb{Z}}$ verifies
\begin{equation}\label{equationcoalescing1}
D^{(n)}_{i}=0\Rightarrow D^{(n)}_{k}=D^{(m)}_{k},\,\,\,\,\,\, \forall \,n\leq m\leq i\leq k.
\end{equation}
This is the \emph{coalescence} property of the sequence of chains $({\bf D}^{(n)})_{n\in\mathbb{Z}}$ illustrated on Figure \ref{fig:combine}.
Using first  (\ref{equationcoalescing1}), and then (\ref{domiprop1}), we obtain that for any $l\geq0$
\[
 \bigcap_{i=-l-1}^{-1}\bigcup_{k=i+1}^{n}\G\{D_{k}^{(i)}=0\D\}=\bigcap_{i=-l-1}^{-1}\bigcup_{k=0}^{n}\{D^{(i)}_{k}=0\}=\bigcup_{k=0}^{n}\{D^{(-l-1)}_{k}=0\}.
\]
It follows from (\ref{stat1}) that for any $l\geq0$ and $n\geq0$
\begin{equation}\label{uhuh}
\G\{\bar{\theta}[0,n]<-l\D\}=\bigcup_{k=0}^{n}\G\{D_{k}^{(-l-1)}=0\D\}.
\end{equation}
Therefore, using the translation invariance of  $\bar{{\bf Z}}$, we obtain
\[
\Prob(\bar{\theta}[0,n]<-l)\leq\sum_{k=0}^{n}\Prob(D_{k}^{(-l-1)}=0)=\sum_{k= l+1}^{l+n+1}\Prob(D_{k}^{(0)}=0).
\]
\end{proof}

\subsection{Study of $\Prob(D_{k}^{(0)}=0)$}\label{2step2}

Define the \emph{inverse} function of  $\bar{\ell}$ by
\begin{equation}\label{r1}
\bar{\ell}^{-1}(i)=\inf\{k\geq1:\bar{\ell}(k)>i\}\,\,,\,\,\forall i\geq0.
\end{equation}
The aim of this subsection is to prove the following Proposition.
\begin{prop}\label{elema}
Let ${\bf D}^{(0)}$ be the chain defined through (\ref{key}) using an $\mathbb{N}$-valued function $\bar{\ell}$ and the i.i.d. chain $\bar{\bf Z}$ on $\{1,\star\}$ with distribution $(\epsilon^{|w|},1-\epsilon^{|w|})$. Then, the sequence $u_{k}:=\mathbb{P}(D^{(0)}_{k}=0)$ 
\begin{enumerate}
\item is summable when $(1-\epsilon^{|w|})^{\bar{\ell}^{-1}(i)}$  is summable,
\item decreases exponentially fast when $(1-\epsilon^{|w|})^{\bar{\ell}^{-1}(i)}$ decreases exponentially fast.
\end{enumerate}
\end{prop}

\begin{proof}
Here we use  the proof given in \citep[][Proposition 2]{bressaud/fernandez/galves/1999a} for the house of cards process. 
Denote by $\zeta$ the first time larger than $0$ such that ${\bf D}^{(0)}$ touches $0$, and by $f_{k}$ the probability $\Prob(\zeta=k)$. 
First of all, we observe that the state $0$ is a renewal state for the chain ${\bf D}^{(0)}$. It follows that 
the sequences $(u_{k})_{k\geq1}$ and $(f_{k})_{k\geq1}$ satisfy
\begin{equation}\label{renewaleq}
u_{k}=\sum_{i=1}^{k}f_{i}u_{k-i}.
\end{equation} 
By (\ref{renewaleq}), the series
\[
F(s):=\sum_{n\geq1}f_{n}s^{n}\,\,\,\textrm{and}\,\,\,U(s):=\sum_{n\geq1}u_{n}s^{n}
\]
are related through
\begin{equation}\label{ooo}
U(s)=\frac{1}{1-F(s)}
\end{equation}
for $s\geq1$ such that $F(s)<1$ \citep[see for example][chap. XIII.10, Theorem 1]{feller/1968}.
In order to prove statement (1), all we need to prove is that the state $0$ is transient, that is $F(1)<1$, whenever $(1-\epsilon^{|w|})^{\bar{\ell}^{-1}(i)}$ is summable. 
%
Suppose that for some $M>0$ we have $\bar Z_{1}^{M}=1^{M}$ so that in particular $D_{M}^{(0)}=M$.
The first possible arrow which can go until or further time $0$ could be the one of  $M+\bar{\ell}^{-1}(M-1)$. This follows from the definition (\ref{r1}) of $\bar{\ell}^{-1}(M-1)$.
 Then, in order that the chain ${\bf D}^{(n)}$ touches $0$ at the first possible time after $M$, it is necessary that $\bar{\ell}^{-1}(M-1)$ stars appear in $\bar{\bf Z}$ from time $M+1$ to time $M+\bar{\ell}^{-1}(M-1)$. 
This is made clear by Figure \ref{fig:combine}. More specifically, we have for any integer $M\geq1$
\begin{equation}\label{yaman}
\bigcup_{i\geq1}\{\zeta=i\}\cap\{\bar Z_{1}^{M}=1^{M}\}=\bigcup_{i\geq M}\{\bar{Z}_{i+1}^{i+\bar{\ell}^{-1}(i-1)}=\star^{\bar{\ell}^{-1}(i-1)}\}\cap\{\bar Z_{1}^{M}=1^{M}\}.
\end{equation}
It follows that using the partition
\[
\bigcup_{i\geq1}\{\zeta=i\}=\bigcup_{i\geq1}\{\zeta=i\}\cap\{\bar Z_{1}^{M}=1^{M}\}\cup \bigcup_{i\geq1}\{\zeta=i\}\cap\{\bar Z_{1}^{M}\neq1^{M}\},\,\,\,\,\forall M\geq1
\]
one obtains the following simple upper bound:
\[
\Prob\Big(\bigcup_{i\geq1}\{\zeta=i\}\Big)\leq \Prob\Big(\bigcup_{i\geq M}\{\bar{Z}_{i+1}^{i+\bar{\ell}^{-1}(i-1)}=\star^{\bar{\ell}^{-1}(i-1)}\}\cap\{\bar Z_{1}^{M}=1^{M}\}\Big)+\Prob\Big(\bar Z_{1}^{M}\neq1^{M}\Big).
\]
The events $\bigcup_{i\geq M}\{\bar{Z}_{i+1}^{i+\bar{\ell}^{-1}(i-1)}=\star^{\bar{\ell}^{-1}(i-1)}\}$ and $\{\bar Z_{1}^{M}=1^{M}\}$ are independents
 since the $\bar Z_{i}$'s are i.i.d. Therefore, for any $M\geq1$
\[
\Prob\Big(\bigcup_{i\geq1}\{\zeta=i\}\Big)\leq\epsilon^{M|w|}\sum_{i\geq M-1}(1-\epsilon^{|w|})^{\bar{\ell}^{-1}(i)}+1-\epsilon^{M|w|}.
\]

If $\sum_{i\geq0}(1-\epsilon^{|w|})^{\bar{\ell}^{-1}(i)}<+\infty$  we can take $M$ sufficiently large to ensure  that $\sum_{i\geq M-1}(1-\epsilon^{|w|})^{\bar{\ell}^{-1}(i)}<1$. Thus $\sum_{i\geq0}(1-\epsilon^{|w|})^{\bar{\ell}^{-1}(i)}<+\infty$ implies $\sum_{i\geq1}f_{i}<1$, concluding the proof of statement (1). 

For the proof of statement (2),  let us suppose that $(1-\epsilon^{|w|})^{\bar{\ell}^{-1}(i)}$ decreases exponentially fast. Then on the one hand  ${\bf D}^{(0)}$ is transient and therefore $F$ and $U$ are related through (\ref{ooo}), and on the other hand  $\bar{\ell}^{-1}(i)\sim i$.
It follows that
\[
f_{n}=\mathbb{P}(D^{(0)}_{i}>0\,,\,\,i=1,\ldots \frac{n}{2}-1)\epsilon^{|w|} (1-\epsilon^{|w|})^{n/2},
\]
and $\frac{f_{n}}{(1-\epsilon^{|w|})^{n/2}}\stackrel{n\rightarrow+\infty}{\longrightarrow}\mathbb{P}(\zeta=+\infty)>0$.
Thus, the radius of convergence of $F$ is 
\[
\lim_{n\rightarrow+\infty}\left((1-\epsilon^{|w|})^{n/2}\right)^{-1/n}=(1-\epsilon^{|w|})^{-1/2}
\] 
which is strictly larger than $1$. Since $F(1)=\mathbb{P}(\zeta<+\infty)<1$, it follows that by continuity, there exists a real number $s_{0}>1$ such that $F(s_{0})=1$. By (\ref{ooo}), this means that $U(s)<+\infty$ for $s<s_{0}$, and by the definition of $U$, it implies that $u_{n}$ decreases faster than $r^{n}$ for $r\in(s_{0}^{-1},1)$.
\end{proof}


\subsection{Proof of the theorem}\label{laprueba}

The proof of the theorem is simple now. 
%
We will show that for any $-\infty< n\leq+\infty$, $\Prob(\theta[0,n]<-l)$ converges to $0$ when $l$ diverges.
Since
\[
\Prob(\theta[0,n]<-l)\leq\Prob\G(\theta\G[0,|w|\G\lceil \frac{n}{|w|}\D\rceil\D]<-|w|\G\lfloor\frac{l}{|w|}\D\rfloor\D)
\]
where $\lfloor r\rfloor$ denotes the integer part of $r$, one obtains by Lemmas  \ref{matador} and \ref{seqgeq1}
\begin{equation}\label{yo}
\Prob(\theta[0,n]<-l)\leq\sum_{k=\G\lfloor\frac{l}{|w|}\D\rfloor }^{\G\lfloor\frac{l}{|w|}\D\rfloor+\G\lceil \frac{n}{|w|}\D\rceil}u_{k}
\end{equation}
for any  $l$ such that $\G\lfloor\frac{l}{|w|}\D\rfloor\geq1$. 
In the conditions of Theorem  \ref{theo1}, we have 
\[
\limsup_{k\rightarrow\infty}\frac{\log\ell^{w}(k)}{\log\left(1/(1-\epsilon^{|w|})\right)^{k/|w|}}<1, 
\]
therefore, there exists a  real number  $\alpha>0$ such that  for any $k$ sufficiently large $\ell^{w}(k)\leq\left(\frac{1}{1-\epsilon^{|w|}}\right)^{k/[|w|(1+\alpha)]}$.
It follows that  $\bar{\ell}(k)\leq\G(\frac{1}{1-\epsilon^{|w|}}\D)^{(k+1)/(1+\alpha)}$ for  $k$ sufficiently large
and by the definition (\ref{r1}) of $\bar{\ell}^{-1}$,  $\bar{\ell}^{-1}(n)\geq \frac{1+\alpha}{\log \frac{1}{1-\epsilon^{|w|}}}\log n-1$ for any $n$ sufficiently large. Therefore there exists $n^{\star}$ such that
\[
\sum_{n\geq0}(1-\epsilon^{|w|})^{\bar{\ell}^{-1}(n)}\leq\sum_{n=0}^{n^{\star}-1}(1-\epsilon^{|w|})^{\bar{\ell}^{-1}(n)}+(1-\epsilon^{|w|})^{-1}\sum_{n\geq n^{\star}}n^{-1-\alpha},
\]
which is finite since $\alpha$ is strictly positive.
By Proposition \ref{elema}, this implies   that $(u_{k})_{k\geq1}$ is summable. It follows by (\ref{yo}) that $\mathbb{P}(\theta[0]<-l)$ is summable in $l$, and  that $\Prob(\theta[0,n]<-l)$ goes to zero when $l$ diverges, for any $n$, $0\leq n\leq+\infty$. 
This concludes the proof of Theorem \ref{theo1}.




\section{Proof of Corollary \ref{coro2}}\label{provatheo1.5}

We refer to  \citep[Section 8]{comets/fernandez/ferrari/2002} for a complete proof of the statements of Corollary \ref{coro2}. The proofs given therein go mainly along the following lines.

\subsection{Existence of a regeneration scheme}
On the one hand, we have that $\Prob(\theta[0,+\infty]=0)>0$, which follows from the fact that
\[
\Prob(\theta[0,+\infty]=0)\geq\Prob(\cap_{i\geq1}\{D^{(0)}_{i}\geq1\}),
\]
which is strictly positive, since the state $0$ is transient in the conditions of Theorem \ref{theo1} (this is shown in the proof of Proposition \ref{elema}).
On the other hand, we have to check that the chain ${\bf \xi}$, defined by $\xi_{j}:=\mathbbm{1}\G\{j=\theta[j,+\infty]\D\}$, is renewal. This follows  from the fact that by the definition (\ref{formaldef}) of the random variable $\theta$, we have
\[
\bigcap_{l=1}^{n}\G\{\theta[t_{l},+\infty]=t_{l}\D\}=\bigcap_{l=1}^{n}\G\{\theta[t_{l},t_{l+1}-1]=t_{l}\D\}
\]
(where we used the convention $t_{n+1}=+\infty$) which is an intersection of independent events, since $\G\{\theta[t_{l},t_{l+1}-1]=t_{l}\D\}$ 
is $\mathcal{F}\G(U_{t_{l}}^{t_{l+1}-1}\D)$-measurable, for $l=1,\ldots,n$. 
To conclude, the fact that the random strings $(\XU_{T_{i}},\ldots,\XU_{T_{i+1}-1})_{i\neq0}$ are i.i.d. follows from the construction using Algorithm 1 or 2.
\subsection{Finite expected size}
To show that the expected size between two consecutive $1$'s in  $\boldsymbol{\xi}$ is finite, we observe that
by stationarity and definition (\ref{formaldef}) of the random variable $\theta$
\[
\mathbb{P}(T_{l+1}-T_{l}\geq m)=\mathbb{P}(\theta[1,+\infty]\leq-m|\theta[0,+\infty]=0)=\mathbb{P}(\theta[0]<-m+1),
\]
which has been proved to be summable in Subsection \ref{laprueba}.

\section{Proof of Theorem \ref{theo2}}\label{proofoftheorem2}

Suppose $(\tau,p)$ satisfies the conditions of Theorem \ref{theo1}. Denote by ${\bf X}$ the unique stationary chain which has been constructed with Algorithm 1 or 2. Define the random variable $L_{i}^{\bf X}:=|c_{\tau}(X_{-\infty}^{i-1})|$
and  denote $\sigma:=\lceil\frac{\ell^{w}(|w|-1)}{|w|}\rceil+1$. For any integers $m,n$ such that $-\infty<m+\sigma|w|\leq n\leq+\infty$ the visible regeneration time of the window $[m,n]$ is 
\begin{equation}\label{thetaX}
\theta^{{\bf X}}[m,n]:=\max\{k\leq t:X_{k}^{k+\sigma|w|-1}=w^{\sigma}\,\textrm{ and }L^{\bf X}_{i}\leq i-k,\,\,i=k+\sigma|w|,\ldots,n\}.
\end{equation}
Observe that although $L_{i}^{\bf X}:=|c_{\tau}(X_{-\infty}^{i-1})|$, the event $\{\theta^{{\bf X}}[m,n]=k\}$ is measurable with respect to the $\sigma$-algebra generated by $X_{k}^{n}$.
To ask for $X_{k}^{k+\sigma|w|-1}=w^{\sigma}$ ensures that there exist realizations of ${\bf X}$ such that $\theta^{\bf X}[m,n]>-\infty$. To see that, observe that we can concatenate one more $w$ to these $\sigma$ consecutive $w$'s without needing to know more than $w^{\sigma}$: for $i=0,\ldots,|w|-2$
\begin{equation}
\begin{array}{cccc}
|c_{\tau}(w^{\sigma-1}\,w)|\leq|c_{\tau}(w^{\sigma-1}\,w\, w_{-|w|}\ldots w_{-|w|+i})|&=&\ell^{w}(i+1)+|w|+i+1\\
&\leq&\ell^{w}(|w|-1)+|w|+i+1\\
& \leq& \sigma|w|+i+1.
\end{array}
\end{equation}

We say that time $t$ is a \emph{visible regeneration time} for the chain ${\bf X}$  if $\theta^{{\bf X}}[t+\infty]=t$.
Finally, define $\boldsymbol{\xi}^{{\bf X}}$ and  ${\bf T}^{{\bf X}}$ using $\theta^{{\bf X}}[t,+\infty]$ in the same way we defined $\boldsymbol{\xi}$ and ${\bf T}$ using $\theta[t,+\infty]$.  What we want to show is that (i) $\theta^{\bf X}[0,+\infty]$ is almost surely finite and (ii) $\boldsymbol{\xi}^{{\bf X}}$ is renewal, with finite expected distance between two consecutive $1$'s. 

We use the sequence ${\bf Z}$ which tells us where we are sure that $w$ occurs in ${\bf X}$. 
The proof of item (i)  is quite similar to the proof of Theorem \ref{theo1}. The main difference 
is in the following new definitions:
\[
L'_{i}:=m_{i}+|w|+\ell^{w}(m_{i}),
\] 
which is always strictly larger than $0$, and for any $n\geq \sigma|w|$
\begin{equation}\label{defthetaprime}
\theta'[0,n]:=\max\{k\leq 0:Z_{k}^{k+\sigma|w|-1}=w^{\sigma}\,\textrm{ and }\,\,L'_{i}\leq i-k\,\,\,\textrm{for}\,\,i=k+\sigma|w|,\ldots,n\}.
\end{equation}
%


We can use the proofs given in Section \ref{proof of main theorem} to show that $\theta'[0,n]\leq\theta^{\bf X}[0,n]$, that $\mathbb{P}(\theta'[0,n]<-l)$ goes to zero as $l$ diverges, for any $n\leq+\infty$ and that $\mathbb{P}(\theta'[0]<-l)$ is summable in $l$.
In order to prove item (ii), we can adapt the proof of \cite{comets/fernandez/ferrari/2002}. Let us denote $v:=w^{\sigma}$ and define for any $-\infty<m\leq n<+\infty$ the events
\[
h[m,n]:=\bigcap_{i=0}^{|w|-1}\{|c_{\tau}(X_{m}^{n}w_{-|w|}\ldots w_{-|w|+i})|\leq n+i+1-m\},
\]
which says that one more $w$ can be concatenate to $X_{m}^{n}$ without needing to look back before time $m$, and
\[
H[m,n]:=\G\{X_{m}^{m+\sigma|w|-1}=v\,, \,\,|c_{\tau}(X_{m}^{i})|\leq i-m\,,\,\,i\in\{m+\sigma|w|,\ldots,n\}\D\}\cap h[m,n].
\]
Both of them are measurable with respect to the $\sigma$-algebra generated by $X_{m}^{n}$. Finally, define
\[
H[m,+\infty]:=\G\{X_{m}^{m+\sigma|w|-1}=v\,, \,\,|c_{\tau}(X_{m}^{i})|\leq i-m\,,\,\,i\geq m+\sigma|w|\D\}
\]
which is measurable with respect to the $\sigma$-algebra generated by $X_{m}^{+\infty}$.
By definition (\ref{thetaX}) of $\theta^{{\bf X}}[m,n]$, $-\infty<m\leq n\leq+\infty$, we have that if 
\begin{equation}\label{sequencet}
t_{1}+\sigma|w|\leq t_{2}+\sigma|w|\leq\ldots\leq t_{n-1}+\sigma|w|\leq t_{n},
\end{equation}
then
\begin{equation}\label{8}
\bigcap_{l=1}^{n}\G\{\theta^{{\bf X}}[t_{l},+\infty]=t_{l}\D\}=\bigcap_{l=1}^{n}H[t_{l},t_{l+1}-1]
\end{equation}
where $t_{n+1}:=+\infty$. This is an intersection of independent events. Then, we observe that by stationarity,
\[
\mathbb{P}(H[j,+\infty])=\mathbb{P}(H[0,+\infty])
\]
and
\[
\mathbb{P}(H[-j,-1])=\mathbb{P}(H[-j,+\infty]|H[0,+\infty])\,,\,\,\,\forall j\geq\sigma|w|.
\]
Together with (\ref{8}), this yields for any sequence of integers $t_{1},\ldots, t_{n}$ verifying (\ref{sequencet})
\[
\mathbb{P}(\xi^{{\bf X}}_{t_{l}}=1,l=1,\ldots,n)=\mathbb{P}(\xi^{{\bf X}}_{0}=1)\prod_{l=1}^{n-1}\mathbb{P}(\xi^{{\bf X}}_{-(t_{l+1})-t_{l}}=1|\xi^{{\bf X}}_{0}=1)
\]
and therefore, the chain $\boldsymbol{\xi}^{\bf X}$ is renewal, proving the existence of the visible regeneration scheme. By stationarity and by definition (\ref{thetaX}) of the random variable $\theta^{\bf X}$, we have  for any $m\geq \sigma|w|$
\[
\mathbb{P}(T^{\bf X}_{l+1}-T^{\bf X}_{l}\geq m)=\mathbb{P}(\theta^{\bf X}[1,+\infty]\leq-m|\theta^{\bf X}[0,+\infty]=0)=\mathbb{P}(\theta^{\bf X}[0]<-m+1),
\]
which is summable in $m$, concluding the proof of Theorem \ref{theo2}.

\section{The complete perfect simulation algorithm, simulation and discussion}\label{realistic}

\subsection{The algorithm}

Algorithm 1 is ``simplistic'' in the sense that in order to compute $\theta[m,n]$, it uses the function $\mathcal{L}$ which is not explicit. A more complete algorithm is given below. We recall that for any $a_{m}^{n}\in A^{n-m+1}$ and $u\in[0,1[$
\[
F(u,a_{m}^{n}):=\sum_{a\in A}a.\mathbbm{1}\{u\in K(a|c_{\tau}(a_{m}^{n}))\}+\star\mathbbm{1}\{u\in [\#\mathcal{E}\epsilon,1[,\, c_{\tau}(a_{m}^{n})=\emptyset\},
\]
where if $m=n+1$, then $c_{\tau}(a_{m}^{n})=c_{\tau}(\emptyset)=\emptyset$. This function contains all the information we need about the probabilistic context tree $(\tau,p)$, and we suppose that it is already implemented in the software used for programing the algorithm.
\begin{algorithm}[h]
\caption{``Explicit'' perfect simulation algorithm of the sample $\XU_{m}^{n}$} 
\begin{algorithmic}[1]
\STATE {\it Input:} $m$, $n$, $F$; {\it Output:} $\theta[m,n]$, $(\XU_{\theta[m,n]},\ldots,\XU_{n})$
\STATE  Sample $U_{m},\ldots,U_{n}$ uniformly in $[0,1[$\\
\STATE $i \leftarrow m$, $B=\{m,\ldots,n\}$, $\theta[m,n] \leftarrow m$, $\XU_{m}^{n}\leftarrow\star^{n-m+1}$\\
\WHILE{$F(U_{i},\XU_{m}^{i-1})\in A$ and $B\neq \emptyset$}                           
\STATE $\XU_{i}\leftarrow F(U_{i},\XU_{m}^{i-1})$\\
\STATE $B\leftarrow B\setminus \{i\}$\\
\STATE $i\leftarrow i+1$
\ENDWHILE\\
\STATE $i\leftarrow m$
\WHILE{$B\neq\emptyset$} 
\STATE $i\leftarrow i-1$\\
\STATE $B\leftarrow B\cup\{i\}$\\
\STATE Sample $U_{i}$ uniformly in $[0,1[$\\
\WHILE {$U_{i}\in[\#\mathcal{E}\epsilon,1[$}
\STATE $i\leftarrow i-1$\\
\STATE $B\leftarrow B\cup\{i\}$\\


\STATE Sample $U_{i}$ uniformly in $[0,1[$\\
\ENDWHILE\\
\STATE $\XU_{i}\leftarrow F(U_{i},\emptyset)$
\STATE $B\leftarrow B\setminus\{i\}$\\
\STATE $t\leftarrow \min B$
\WHILE{$F(U_{t},\XU_{i}^{t-1})\in A$ and $B\neq\emptyset$}
\STATE $\XU_{t}\leftarrow F(U_{t},\XU_{i}^{t-1})$\\
\STATE $B\leftarrow B\setminus\{t\}$\\
\STATE $t\leftarrow \min B$
\ENDWHILE\\
\ENDWHILE\\
\STATE $\theta[m,n]\leftarrow i$
\RETURN $\theta[m,n]$, $(\XU_{\theta[m,n]},\ldots, \XU_{n})$
\end{algorithmic}
\end{algorithm}
\begin{figure}[htp]
\centering
\includegraphics{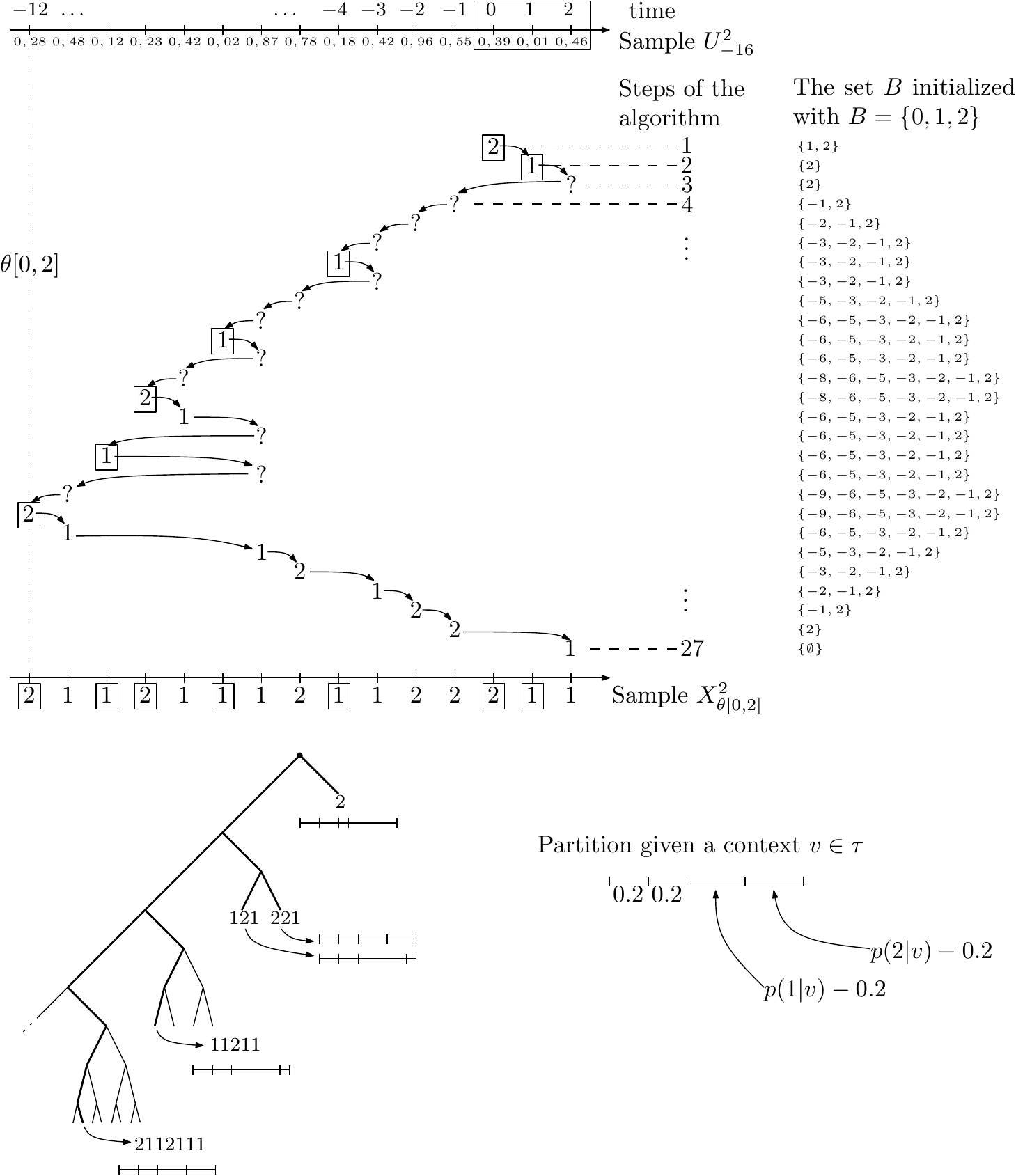}
\caption{An explicit perfect simulation, using a sample $U_{-12}^{2}$. The transition probabilities of the context tree are given in the bottom of the Figure. We refer to Figure \ref{fig:partitionof[0,1]} for the reader to recall what the small intervals represent. Both symbols $1$ and $2$ are $0.2$-regulars. The probability transitions are: $p(2|2)=0.7$, $p(2|121)=0.3$, $p(2|122)=0.5$, $p(2|11211)=0.3$ and $p(2|1112112)=0.5$}
\label{fig:algorithm2}
\end{figure}
\begin{figure}[h!]
\begin{center}
\includegraphics[scale=0.7]{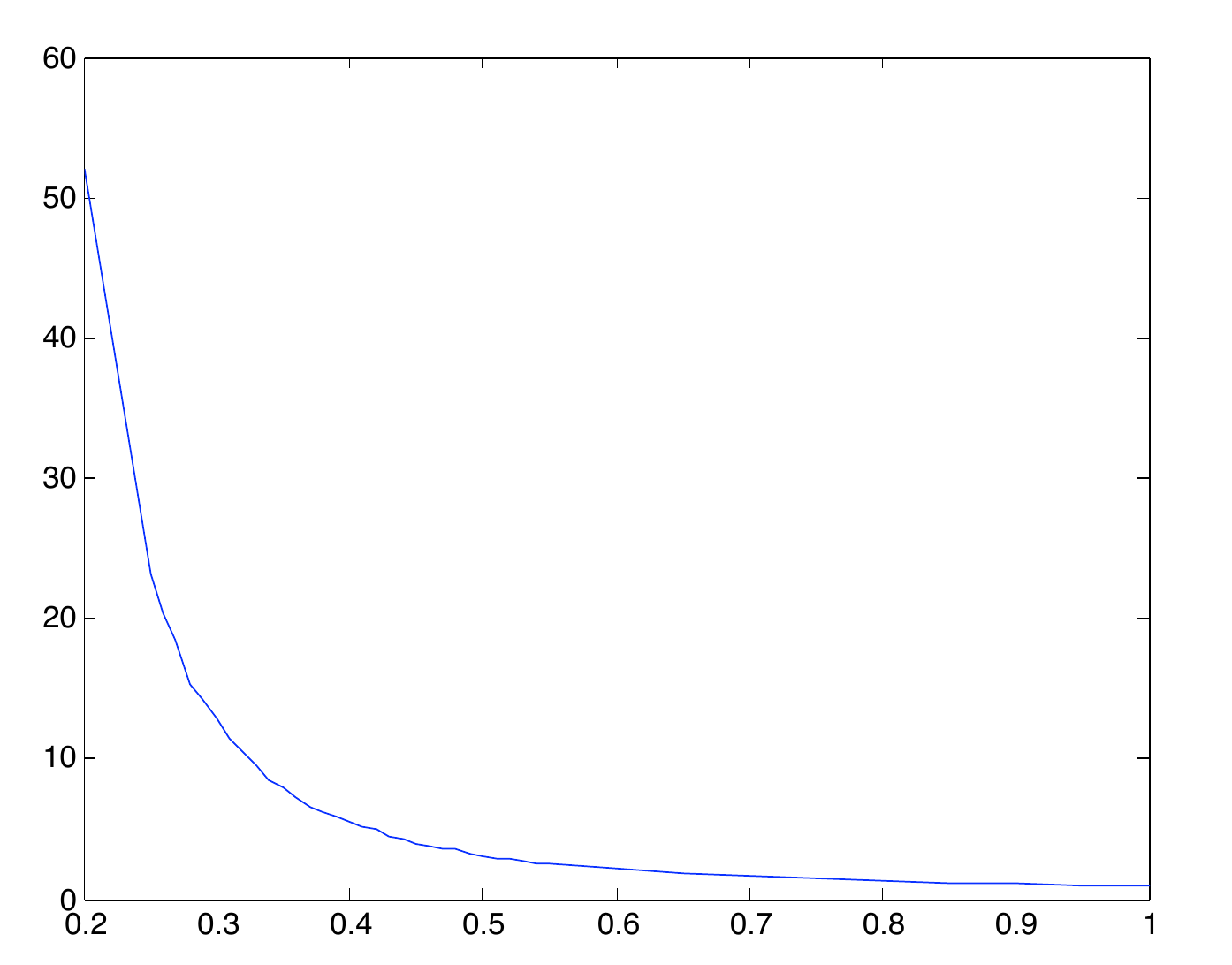}
\caption{Graph representing the influence of the value of $\epsilon$ on the quantity $\mathbb{E}|\theta[0]|$. The $x$-axis represents the successives values of $\epsilon$, from $0.2$ to $1$. The $y$-axis represents $\mathbb{E}|\theta_{\epsilon}[0]|$, the expected value of $\theta[0]$, when this later is computed using value $\epsilon$.}
\label{fig:epsilondobom}
\end{center}
\end{figure}

The algorithm uses two variables: $i$ which is a time index and $B$ which is a set of time indexes. 
The set $B$ keeps track of the set of sites which have to be constructed. It is initialized with $B=\{m,\ldots,n\}$, which is the set of time indexes  to be constructed, and the algorithm terminates when $B=\{\emptyset\}$.
In the first ``while'' loop (lines 2 to 8), we sample $U_{m}^{n}$, and directly attempt to construct $\XU_{m}^{n}$ using this information. If the algorithm manages to do this, it returns $\theta[m,n]=m$ and the constructed sample $\XU_{m}^{n}$. Otherwise, it enters the second while loop (lines 10 to 27). In this loop, each time the algorithm cannot construct the next site of $B$, it generates a new uniform random variable backward in time. At each  new generated random variable, the algorithm attempts to go as far as possible in the construction of the remaining sites of $B$ using the uniform that have been previously generated.

\subsection{Simulation}

We will say that the algorithm makes a step each time it ``enters''  a ``while'' loop. On Figure \ref{fig:algorithm2}, it corresponds to the number of arrows, plus one. The total number of steps $N[m,n]$ needed for the construction of a sample $X_{m}^{n}$ is
\[
N[m,n]=(n-m+1)+2\times(m-\theta[m,n]).
\]
Let us denote by $C$ the maximum number of operations the algorithm need in order to make a step. Suppose we want to construct a sample $X_{0}^{n-1}$, then the expected number of operations is bounded above by 
\[
C\times \left(n+2\times \mathbb{E}|\theta[0,n-1]|\right).
\]
Figure \ref{fig:algorithm2} illustrates an explicit perfect simulation of this chain using a finite sample of  ${\bf U}$, in the case where $\epsilon=0,2$ and both symbols are $\epsilon$-regular.

The results of the present paper tell us that in the conditions of Theorem \ref{theo1}, this expectation is finite. However, no insight is given on how large it can be. This is due to the fact that we did not manage to obtain sufficiently good explicit bounds for the probability of return to $0$ of the chain $\bf{D}^{(0)}$. This is also the case of \cite{comets/fernandez/ferrari/2002}. Anyway, this bound should depends strongly on the parameter $\epsilon$, as in \cite{comets/fernandez/ferrari/2002} (see (2.4), (2.5) therein, their $a_{0}$ corresponds basically to our $\#\mathcal{E}\epsilon$). In absence of such bounds, we implemented the above pseudo-code in the case of the context tree of Section \ref{anexample}. We assume  that $p(2|v)=\epsilon$ for any $v\in\tau$. Notice that this case corresponds to the i.i.d. chain with probability $\epsilon$ to get the symbol $2$. This assumption simplifies considerably the implementation of the algorithm and  gives us the largest possible regeneration times within the class of probabilistic context trees for which the symbol $2$ is $\epsilon$-regular. 
We used increasing values of $\epsilon$, from $0,2$ to $1$, and for each value, we made the mean over $10,000$ iterations of Algorithm 2.  The resulting graph is given in Figure \ref{fig:epsilondobom}. We can derive the corresponding expected number of steps $\mathbb{E}N[0]$ realized by the algorithm, and the expected number of operations too.

\section{Final comments and references}

The first study of chains of infinite order seems to come back to the seminal papers of \citet{onicescu/mihoc/1935}. They called these chains \emph{``cha\^ines \`a liaison compl\`etes''} (chains with complete connections).
Then \citet{doeblin/fortet/1937} proved the results on speed of convergence towards the invariant measure under the  continuity conditions. We mention without further details some works as, for example,  \citet{harris/1955}, \citet{lalley/1986}, \citet{berbee/1987}, \citet{bressaud/fernandez/galves/1999a}, \citet{johansson/oberg/2003} among others. We refer to the book of \citet{iosifescu/grigorescu/1990} for a complete review of the area, and to the book of \citet{fernandez/ferrari/galves/2001} for an introduction to the constructive approach. 

Chains with variable length have been introduced by \cite{rissanen/1983} as a universal model for data compression. It has been shown to have a great applicability in statistical inference and modeling.  For a review and reference in this area, we refer to the paper by \citet{galves/eva/2008}.


\vspace{0.1cm}
On perfect simulation using the CFTP method, we refer to the webpage of Prof. David Bruce Wilson: {\ttfamily http://dbwilson.com/exact/}.
 The reader will find therein an extensive list of publications in the area. A very interesting issue, related to our work  is whether or not the $\epsilon$-regularity (or the weakly non-nullness) assumption is necessary for the existence of a (not necessary practical) CFTP algorithm. We mention that necessary conditions exist for  Markov chains: \cite{foss/tweedie/1998} shown that such a coupling exists if and only if the Markov chain is geometrically ergodic. 

\section*{Acknowledgements}

I am thankful to my PhD advisor Prof. A. Galves for his invaluable remarks and help.

\bibliography{sandro_bibli}

%






\end{document}